\documentclass[12pt,reqno,a4paper]{amsart}

\usepackage{verbatim, amssymb, enumitem}
\usepackage[
  colorlinks=true,
  pdfencoding=auto, 
psdextra]{hyperref}

\renewcommand \a{\alpha}
\renewcommand \b{\beta}
\newcommand \K{\delta}
\newcommand \la{\lambda}

\newcommand \id{\mathrm{id}}
\newcommand \br{\mathbb{R}}
\newcommand \bc{\mathbb{C}}

\newcommand \rk{\operatorname{rk}}
\newcommand \Span{\operatorname{Span}}
\newcommand \Ker{\operatorname{Ker}}

\newcommand \End{\operatorname{End}}
\newcommand \Ric{\operatorname{Ric}}
\newcommand \ric{\operatorname{ric}}
\newcommand \Tr{\operatorname{Tr}}
\newcommand \db{\partial}

\newcommand\g{\mathfrak g}
\newcommand\h{\mathfrak h}

\newcommand\sg{\mathfrak{s}}

\newcommand \gn{\mathfrak{n}}
\newcommand \gm{\mathfrak{m}}

\newcommand \ad{\operatorname{ad}}
\newcommand \diag{\operatorname{diag}}
\renewcommand \G{\Gamma}
\newcommand \n{\nabla}

\newcommand \Om{\Omega}
\renewcommand \div{\operatorname{div}}
\newcommand \grad{\operatorname{grad}}
\newcommand \scal{\operatorname{scal}}

\newcommand \<{\langle}
\renewcommand \>{\rangle}
\newcommand \ip{\langle \cdot, \cdot \rangle}
\newcommand \ipb{\overline{\langle \cdot, \cdot \rangle}}

\theoremstyle{plain}
\newtheorem{theorem}{Theorem}
\newtheorem*{theorem*}{Theorem}
\newtheorem{corollary}{Corollary}
\newtheorem*{conj*}{Conjecture}
\newtheorem{lemma}{Lemma}

\newtheorem*{prop*}{Proposition}

\theoremstyle{definition}
\newtheorem{definition}{Definition}
\newtheorem*{definition*}{Definition}

\theoremstyle{remark}
\newtheorem{remark}{Remark}
\newtheorem{example}{Example}

\begin{document}

\title{Einstein extensions of Riemannian manifolds}

\author{D.Alekseevsky}
\address{IITP, Russian Academy of Sciences, Moscow 127051, Russia}
\email{dalekseevsky@iitp.ru}

\author{Y.Nikolayevsky}
\address{Department of Mathematics and Statistics, La Trobe University, Melbourne 3086, Australia}
\email{y.nikolayevsky@latrobe.edu.au}

\thanks{The authors were partially supported by ARC Discovery Grant DP130103485.}



\begin{abstract}
Given a Riemannian space $N$ of dimension $n$ and a field $D$ of symmetric endomorphisms on $N$, we define the extension $M$ of $N$ by $D$ to be the Riemannian manifold of dimension $n+1$ obtained from $N$ by a construction similar to extending a Lie group by a derivation of its Lie algebra. We find the conditions on $N$ and $D$ which imply that the extension $M$ is Einstein. In particular, we show that in this case, $D$ has constant eigenvalues; moreover, they are all integer (up to scaling) if $\det D \ne 0$. They must satisfy certain arithmetic relations which imply that there are only finitely many eigenvalue types of $D$ in every dimension (a similar result is known for Einstein solvmanifolds). We give the characterisation of Einstein extensions for particular eigenvalue types of $D$, including the complete classification for the case when $D$ has two eigenvalues, one of which is multiplicity free. In the most interesting case, the extension is obtained, by an explicit procedure, from an almost K\"{a}hler Ricci flat manifold (in particular, from a Calabi-Yau manifold). We also show that all Einstein extensions of dimension four are Einstein solvmanifolds. A similar result holds valid in the case when $N$ is a Lie group with a left-invariant metric, under some additional assumptions.
\end{abstract}

\maketitle

\section{Introduction}
\label{s:intro}

The construction and the study of Einstein manifolds is one of the main avenues of Riemannian Geometry. One of the starting points of our paper is the theory of Einstein homogeneous manifolds of negative scalar curvature. Assuming \emph{Alekseevsky Conjecture} (and the fact that the isometry group is linear), such manifolds are necessarily solvmanifolds, solvable Lie groups with a left-invariant Einstein metric. At present, the theory of Einstein solvmanifolds is very well developed \cite{Lsurv}. The basic construction is as follows. At the level of Lie algebras, one starts with a nilpotent Lie algebra $\gn$ with a special \emph{nilsoliton inner product} characterised by the property that its Ricci operator is a linear combination of the identity operator and the \emph{Einstein derivation} $D$. The derivation $D$ is always symmetric and its eigenvalues, up to scaling, are natural numbers (not every nilpotent Lie algebra admits such a derivation and such an inner product; those which do are called \emph{nilsolitons}). The rank one extension of $\gn$ by $D$ is a solvable Lie algebra $\sg$. Extending the inner product from $\gn$ to $\sg$ in such a way that the extension is orthogonal (and choosing the correct scaling factor) one obtains a metric Einstein solvable Lie algebra whose solvable Lie group, with the corresponding left-invariant metric, is an Einstein solvmanifold $M$. All rank one Einstein extensions can be obtained in this way and the higher rank extensions can be obtained from rank one extensions by a known procedure \cite[Theorem~4.18]{Heb} and \cite{Lstand}. One can see that the resulting Riemannian metric on $M$ has precisely the form as in the definition below. 

The main idea of this paper is to drop the homogeneity assumption and to construct rank one Einstein extensions of arbitrary Riemannian manifolds by a field of symmetric endomorphisms $D$, as described below.

\begin{definition} \label{def:main}
Let $(N, g)$ be a Riemannian manifold of dimension $n > 1$, and $D$ a field of symmetric endomorphisms on $(N, g)$. For $u \in \br$, the \emph{$D$-deformation} of the metric $g$ on $N$ is the metric on $N$ given by $g^u:=(\exp(uD))^*g$. The \emph{$D$-extension} is the Riemannian manifold $(M, g^D)$ given by
\begin{equation*}
  (M:=\br \times N, \; g^D:= du^2 + g^u).
\end{equation*}
\end{definition}

When $D$ has eigenvalues $q_1, \dots, q_m$ of constant multiplicities and $V(q_i)$ are the corresponding eigendistributions, the $D$-deformation is given by $g^u = e^{2q_1u} g_1 + \dots + e^{2q_mu} g_m$, and the $D$-extension, by
\begin{equation*}
g^D = du^2 + g^u = du^2 + e^{2q_1u} g_1 + \dots + e^{2q_mu} g_m,
\end{equation*}
where $g_i:= g_{|V(q_i)}$. Clearly, $D$ remains symmetric with respect to all the metrics $g^u$ on $N$.

This construction, both in the Riemannian and pseudo-Riemannian cases, is known in the literature (see e.g., \cite{Herv}) and also appears in the theory of Riemannian submersions \cite[Chapter~9]{Bes}. It is a generalisations of the warped product metrics; note however that we make no assumptions on the integrability of the eigendistributions of $D$.

\begin{definition} \label{def:stable}
A manifold $(N, g)$ (and the metric $g$ on the manifold $N$) is called \emph{Ricci $D$-stable} if the Ricci operator $\Ric^u := \Ric_{g^u}$ does not depend on $u$, and is called \emph{$D$-Einstein} if the extension $(M,g^D)$ is Einstein.
\end{definition}

Our main question is, when a metric is $D$-Einstein, or in other words, \emph{when the extension $(M, g^D)$ is Einstein}? As we will see, in many cases, this general construction bears a remarkable resemblance to the homogeneous (the solvmanifold) case, and in some cases (as in Theorem~\ref{th:dim3}), the Einstein condition even \emph{implies} the homogeneity. Below we present the structure of the paper and the main results.

\smallskip

In Section~\ref{s:ric}, we compute the Ricci tensor of $(M, g^D)$ and prove the following theorem which gives a necessary and sufficient condition for a Riemannian manifold $(N,g)$ to admit an Einstein $D$-extension.

\begin{theorem}\label{th:Dconst}
Let $(M, g^D)$ be the $D$-extension of $(N,g)$. Then $(M, g^D)$ is Einstein if and only if the following two conditions are satisfied:
\begin{enumerate}[label=\emph{(\alph*)},ref=\alph*]
  \item \label{it:Dconst1}
  The endomorphism $D$ has constant eigenvalues and
  \begin{equation} \label{eq:ein1div}
  \div D = 0,
  \end{equation}
  where $\div$ is the divergence relative to $g$ \emph{(}so that $(\div D)X=\Tr (Y \mapsto (\nabla_Y D)X)$\emph{)}.

  \item \label{it:Dconst2}
  The manifold $(N,g)$ is Ricci $D$-stable and
  \begin{equation} \label{eq:ein1Ru}
    \Ric^u= (\Tr D) \, D - \Tr (D^2) \, \id.
  \end{equation}
  \end{enumerate}
The Einstein constant of $g^D$ is $- \Tr (D^2)$.
\end{theorem}


\begin{example} \label{ex:id}
A Ricci flat manifold $(N^n, g)$ is $\id$-Einstein, i.e. the metric $g^{\id} = du^2 + e^{2u}g$ is Einstein with the Einstein constant $-n$. In particular, if $g$ is Euclidean, then $g^{\id}$ is a hyperbolic metric written in horospherical coordinates. The converse (``any $\id$-Einstein manifold is Ricci flat") follows from~\eqref{eq:ein1Ru}.
\end{example}

\begin{example} \label{ex:product}
A direct product $(N_1 \times N_2, g_1+g_2)$ of Ricci $D_i$-stable manifolds $(N_i, g_i), \; i = 1, 2$, is Ricci $D$-stable, where $D=D_1 \oplus D_2$. Moreover, it is $D$-Einstein if and only if for $i = 1, 2$, condition~\eqref{it:Dconst1} of Theorem~\ref{th:Dconst} is satisfied and $\Ric_{g_i} = (\Tr D) D_i - \Tr(D^2) \id_{N_i}$.
\end{example}

\begin{example} \label{ex:group}
Let $(N, g)$ be a Lie group with a left-invariant metric and $D$ be defined by a symmetric derivation of the Lie algebra of $N$. Then $(N, g)$ is $D$-stable (see Section~\ref{s:homo}). 
\end{example}


In Section~\ref{s:eigen} we study the \emph{eigenvalue type of $D$}, the vector $\mathbf{p}=(p_1, \dots, p_n)^t$ of its eigenvalues (recall that all of them must be constant by Theorem~\ref{th:Dconst}). We call $\mathbf{p}$ \emph{the spectral vector}. We show in Lemma~\ref{l:scalar} that $D$ is scalar if and only if $(N, g)$ is Ricci flat, and that in all the other cases, the eigenvalues satisfy some restrictions, in particular, some nontrivial relations of the form $p_k=p_i+p_j$.

In the case when all $p_i$ are nonzero, we can be much more specific. In the Euclidean space $\br^n$ with an orthonormal basis $f_i, \; i=1, \dots n$, introduce the subset $F=\{f_i+f_j-f_k \, : \, i \ne j, \, k \ne i,j\}$ and consider the vectors $\mathbf{p}=(p_1, \dots, p_n)^t$ and $\mathbf{1}_n=(1,\dots,1)^t$ ($n$ ones). In the (finite) set $F \cap \mathbf{p}^\perp$, choose a maximal linearly independent subset $F_{\mathbf{p}}=\{v_1, \dots, v_m\}$ (any of them, if there are more than one). Let $V$ be an $n \times m$ matrix whose vector columns are the vectors $v_a$ (so that if $v_a=f_i+f_j-f_k \in F_{\mathbf{p}}$, then the $a$-th column of $V$ has a one in the $i$-th and in the $j$-th rows, a minus one in the $k$-th row, and zeros elsewhere). The following theorem is analogous to \cite[Theorem~4.14, Corollary~4.17]{Heb} for rank one Einstein solvmanifolds.

\begin{theorem}\label{th:nonzero}
Let $(M,g^D)$ be an Einstein $D$-extension of $(N,g)$ with $\det D \ne 0$.
\begin{enumerate}[label=\emph{(\roman*)},ref=\roman*]
  \item \label{it:nonzero1}
  Then the projection of $\mathbf{1}_n$ to $\mathbf{p}^\perp$ belongs to the convex cone hull of the set $F \cap \mathbf{p}^\perp$.
  \item \label{it:nonzero2}
  If, in addition, $\Tr D \ne 0$, then up to scaling, the spectral vector is given by
  \begin{equation}\label{eq:eigenstr}
    \mathbf{p}=\mathbf{1}_n-V(V^tV)^{-1}\mathbf{1}_m.
  \end{equation}
  In particular, up to scaling, all the $p_i$ are integers. Moreover, for every dimension $n$, there is only a finite number of the eigenvalue types of the operators $D$, with $\det D \ne 0$ and $\Tr D \ne 0$.
\end{enumerate}
\end{theorem}

If the operator $D$ is non-scalar, the simplest possible case to consider is when it has an eigenvalue of multiplicity $n-1$, so that $\mathbf{p}=(\lambda, \dots, \lambda, \nu)^t, \; \la \ne \nu$. Up to scaling, we can have $(\la, \nu)=(0,1), (1,0)$, or $(1,2)$, where the fact that $\nu/\la=2$ in the latter case follows from Theorem~\ref{th:nonzero} (or from Lemma~\ref{l:scalar}\eqref{it:sc2} below). These three eigenvalue types are studied in Section~\ref{s:Heis}. If $\mathbf{p}=(0, \dots, 0, 1)^t$, we have the following theorem. 

\begin{theorem}\label{th:01}
Let $(M,g^D)$ be the $D$-extension of $(N,g)$ with the spectral vector $\mathbf{p}=(0, \dots, 0, 1)^t$. If $(M,g^D)$ is Einstein, then $(N,g)$ is locally isometric to the Riemannian product of the real line and an Einstein manifold $N'$ of dimension $n-1$ with the Einstein constant $-1$. The manifold $(M,g^D)$ is locally isometric to the Riemannian product of the hyperbolic plane of curvature $-1$ and $N'$.
\end{theorem}

When $\mathbf{p}=(1, \dots, 1, 0)^t$ we show that $(M, g^D)$ is a warped product with a two-dimensional base, and that it can be obtained as a ``double extension", by two commuting extensions, of a Ricci flat manifold $N'$ (Theorem~\ref{th:10}).

In the third case (which is probably the most interesting one), we prove the following theorem.

\begin{theorem}\label{th:12}
Let $(M,g^D)$ be the $D$-extension of $(N,g)$ with the spectral vector $\mathbf{p}=(1, \dots, 1, 2)^t$. The the following are equivalent.
\begin{enumerate}[label=\emph{(\roman*)},ref=\roman*]
  \item \label{it:12ein}
  The extension $(M,g^D)$ is Einstein.

  \item \label{it:12Keta}
  The Riemannian manifold $(N,g)$ is locally K-contact $\eta$-Einstein with $\ric= -2 g + (n+1) \eta \otimes \eta$, where $\eta$ is the one-form dual to a unit eigenvector $\xi$ of $D$ with the eigenvalue $2$.

  \item \label{it:12KahlerRf}
  The metric $g$ on $N$ is locally given by $g = g' + (dt+\theta')^2$, where $t \in \br$ and $(N',g')$ is an almost K\"{a}hler, Ricci flat manifold and $\theta'$ is a $1$-form on $N'$ such that $d\theta'=\omega$, the K\"{a}hler form on $N'$.
\end{enumerate}
Under any of the above three conditions, the Einstein metric $g^D$ on $M$ is almost K\"{a}hler and is locally given by $g^D=du^2+e^{2u}g'+e^{4u}(dt+\theta')^2$ in the notation of \eqref{it:12KahlerRf}.
\end{theorem}


We prove Theorem~\ref{th:12} in Section~\ref{ss:12}. Recall that a unit Killing vector field $\xi$ on a Riemannian manifold $(N,g)$ defines a \emph{$K$-contact structure} if $\xi^{\perp}$ is a contact distribution with the contact form $\eta = g\circ \xi$ and the restriction $J = \Phi_{|\xi^{\perp}}$ of the endomorphism $\Phi=-\nabla \xi$ to $\xi^{\perp}$ is an almost CR structure, that is, $J^2 = -\id$. It is called a \emph{Sasakian structure} if $(\xi^{\perp}, J)$ is a CR structure or equivalently, if $(\nabla_X \Phi)= \xi \otimes g \circ X - X \otimes g \circ \xi$, and it is called \emph{$\eta$-Einstein}, if $\ric = ag + b\eta \otimes \eta $ for some constants $a,b \in \br$ \cite{Blair}.

\begin{remark}\label{rem:contact}
Note that in \eqref{it:12Keta}, ``locally" can be replaced by ``globally" if $N$ is orientable (or otherwise we can replace $N$ by its orientable cover). We also note that the equivalence of \eqref{it:12Keta} and \eqref{it:12KahlerRf} is a known fact \cite{Blair, BGbook}, which we included for completeness; the almost K\"{a}hler, Ricci flat manifold in $(N',g')$ in \eqref{it:12KahlerRf} is locally the ``leaf space" of the geodesic foliation on $(N,g)$ defined by $\xi$.

If we additionally assume $N$ to be compact (and orientable), then by the result of \cite[Theorem~7.2]{BG}, $(N,g)$ is Sasakian, and then $(N',g')$ is K\"{a}hler and Ricci flat and hence $(M,g^D)$ is Einstein K\"{a}hler. Moreover, if the geodesic foliation on $N$ defined by $\xi$ is \emph{regular} (this means that every point has a neighbourhood through which every geodesic of that foliation passes at most once), one obtains that $(N,g)$ is a circle bundle over a Calabi-Yau manifold $(N',g')$.

On the other hand, starting with an almost K\"{a}hler non-K\"{a}hler Ricci flat manifold $(N',g')$ (an example is constructed in \cite{NP}) one gets that $(N,g)$ is K-contact, but is not Sasakian.
\end{remark}

\begin{remark}\label{rem:deform}
It is interesting to compare our construction to the standard cone construction in contact geometry. Let $(N,g, \eta)$ be a contact metric structure
(this means that $\eta$ is a contact form such that the associated Reeb vector field $\xi$ is unit and $g^{-1} \circ d \eta_{|\Ker \eta}$ is an almost CR structure on $\Ker \eta$). Then the cone $C = \br_+ \times N$ with the metric $g_C = du^2 + u^2 g, \; u \in \br_+$, is an almost K\"ahler manifold with the K\"ahler form $\omega = d (\frac12 u^2 \eta)= u du \wedge \eta + \frac12 u^2 d \eta$. It is a K\"ahler manifold if and only if $(N,g, \eta)$ is a Sasakian manifold.

Our construction can be viewed as the ``$\mathcal{D}$-deformation cone". For a contact metric structure $(N,g, \eta)$ and a positive number $r > 0$, the $\mathcal{D}$-deformed structure is given by $\tilde{\eta}=r\eta, \, \tilde{\xi}=r^{-1} \xi$ and $\tilde{g} = rg + r(r - 1) \eta \otimes \eta$; see \cite[\S7.3]{Blair}, \cite[\S7.3.3]{BGbook}. A $\mathcal{D}$-deformation preserves the property of being K-contact, Sasakian and $\eta$-Einstein. Our metric $(N,g^u)$ is obtained from $(N,g)$ by the $\mathcal{D}$-deformation with $r=e^{2u}$.
\end{remark}

\smallskip

In Section~\ref{s:dim3}, we consider the case $n=3$, the first case when our construction produces interesting examples. Remarkably, in that case the Einstein condition forces the homogeneity.

\begin{theorem}\label{th:dim3}
Let $(M,g^D)$ be the $D$-extension of the manifold $(N,g)$ of dimension $3$. If $(M,g^D)$ is Einstein, then both $(N,g)$ and $(M,g^D)$ are locally isometric to Lie groups with left-invariant metrics; $N$ is a nilmanifold or a solvmanifold, $D$ is a derivation, and $(M,g^D)$ is an Einstein solvmanifold. All the possible cases, up to scaling, are listed in Table~\ref{t:d3}.
\begin{table}[h] 
\renewcommand{\arraystretch}{1.2}
\begin{tabular}{|c|c|c|c|}
  \hline
  $p_i$ & $\gn$ & $(M, g^D)$ & $ds^2$ \\
  \hline
  $0,0,0$ & \emph{abelian} & $\br^4$ & $du^2+(dx^1)^2+(dx^2)^2+(dx^3)^2$ \\
  \hline
  $1,1,1$ & \emph{abelian} & $H^4(-1)$ & $du^2+e^{2u}((dx^1)^2+(dx^2)^2+(dx^3)^2)$ \\
  \hline
  $1,1,2$ & \parbox[c][1.2cm]{2.5cm}{\emph{Heisenberg,}\\ $[\overline{e}_1,\overline{e}_2]= 2\overline{e}_3$} & $\bc H^2 (-4)$ & \parbox[c][1.2cm]{6.8cm}{$\raggedright du^2+e^{2u}((dx^1)^2+(dx^2)^2)$ \\ \raggedleft $+e^{4u}(dx^3+x^1 dx^2-x^2 dx^1)^2$} \\
  \hline
  $1,p,0$ & \parbox[c][1.7cm]{2.5cm}{\emph{Solvable,} \\ $[\overline{e}_3,\overline{e}_1]= p \overline{e}_1$, \\ $[\overline{e}_3,\overline{e}_2]=- \overline{e}_2$} & \parbox[c][1.5cm]{3cm}{\raggedright $H^2(-(p^2+1))$ \\ \raggedleft $\times H^2(-(p^2+1))$} & $du^2+(dx^3)^2+e^{2 (u- p x^3)}(dx^1)^2+e^{2 (p u+ x^3)}(dx^2)^2$ \\
  \hline
\end{tabular}
\vspace{.2cm}
\caption{Four-dimensional Einstein extensions.}
\label{t:d3}
\end{table}
\end{theorem}

In the first column of Table~\ref{t:d3}, we list the eigenvalues $p_i$ of the derivation $D$ of the Lie algebra $\gn$ of $N$, with the corresponding eigenvectors $\overline{e}_i, \, i=1,2,3$. The second column gives the types and the defining relations for $\gn$ (note that the relations for the Lie algebra $\g$ of $M$ are obtained by adding the relations $[e_4,\overline{e}_i]=p_i \overline{e}_i$ to the relations for $\gn$). The third and the fourth columns give the homogeneous spaces to which $(M,g^D)$ is locally isometric and the explicit forms of the metric $(M,g^D)$ in local coordinates respectively, where we denote $H^m(c)$ the hyperbolic space of curvature $c$, and $\bc H^2 (-4)$ the complex hyperbolic space of holomorphic curvature $-4$. 

Note that in the last row of Table~\ref{t:d3}, $p$ is arbitrary. The fact that the metric $g^D$ on $M$ is indeed the Riemannian product of two hyperbolic planes (of curvature $-(p^2+1)$ each) can be seen by the change of variables $y_1=(p^2+1)^{-1/2}(u-p x^3), \; y_2=(p^2+1)^{-1/2}(pu+x^3)$.

\smallskip

In Section~\ref{s:homo}, we turn our attention to the case when $N$ is a Lie group with a left-invariant metric $g$ and $D$ is left-invariant. Denote $\gn$ the Lie algebra of $N$. It would be interesting to know if the condition that $(M,g^D)$ is Einstein forces it to be an Einstein solvmanifold. We answer this question in positive in two cases. 

\begin{theorem}\label{th:group}
Suppose the extension $(M,g^D)$ of a Lie group $(N,g)$ by $D$ is Einstein, and that both $g$ and $D$ are left-invariant. 
\begin{enumerate}[label=\emph{(\alph*)},ref=\alph*]
  \item \label{it:bpos}
  If the Killing form of $\gn$ is nonnegative, then $D$ is a symmetric derivation of $\gn$. 

  \item \label{it:decomphom}
  Let $\gn=\h \oplus \gm$ be a decomposition into an orthogonal sum of an abelian subalgebra $\h$ and a nilpotent ideal $\gm$ and suppose that one of the following conditions is satisfied:
  \begin{enumerate}[label=\emph{(\roman*)},ref=\roman*]
    \item \label{it:decompD}
    either $D$ preserves the decomposition $\gn=\h \oplus \gm$,
    \item \label{it:ni1ab}
    or $\gm$ is abelian.
  \end{enumerate}
  Then there exists a metric solvable Lie group $(N',g')$ and an isometry $\Phi: (N,g) \to (N',g')$ such that $D'=(d\Phi)D$ is left-invariant relative to $N'$ and $D'(e)$ is a symmetric derivation of the Lie algebra $\gn'$ of $N'$.
\end{enumerate}
In both cases, the manifold $(M, g^D)$ is an Einstein solvmanifold. 
\end{theorem}

\section{Ricci curvature of $D$-deformation and $D$-extension. Proof of Theorem~\ref{th:Dconst}}
\label{s:ric}

\subsection{Ricci curvature of the $D$-extension}
\label{ss:ricDext}

In this subsection, we find the Ricci tensor of the $D$-extension $(M, g^D)$, assuming neither $D$-stability, nor $D$-Einstein property.

Locally, on an open, connected domain $U$ of $(N, g)$ on which the eigenvalues $p_i$ of $D$ have constant multiplicities, we can choose a smooth orthonormal frame $\overline{e}_i, \; i=1, \dots, n$, of eigenvectors of $D$. We extend $\overline{e}_i$ and $D$ to $M':=\br \times U \subset M$ by the Lie translation along the vector field $\partial_u$. Let $\overline{\theta}^i$ be the one-forms dual to $\overline{e}_i$. The metric $g^D$ on $M'$ is given by
\begin{equation}\label{eq:metricM}
ds^2=du^2+\sum\nolimits_{i=1}^n e^{2up_i}(\overline{\theta}^i)^2.
\end{equation}

Let $e_a, \; a=0, \dots , n$, be the orthonormal frame on $(M', g^D)$ given by $e_0=\partial_u, \; e_i = e^{-up_i}\overline{e}_i$ for $i >0$, and let $\theta^a$ be the $1$-forms dual to $e_a$, so that 
\begin{equation*}
\theta^0=du, \qquad \theta^i=e^{up_i}\overline{\theta}^i, \; \text{for }i >0.
\end{equation*}

The structure equations for $(M, g^D)$ are given by
\begin{equation}\label{eq:structure}
d \theta^a = - \sum\nolimits_b \psi^a_b \wedge \theta^b,\; \psi^a_b =- \psi^b_a, \qquad \Om^a_b = d \psi^a_b + \sum\nolimits_c \psi^a_c \wedge \psi^c_b,
\end{equation}
where $\psi_a^b$ and $\Om_a^b$ are the connection and the curvature two-forms respectively. Decomposing them by the basis $\theta^a \wedge \theta^b$ we obtain
\begin{equation*} 
\psi_b^a = \sum\nolimits_c \G_{bc}^a \theta^c, \; \; \G_{bc}^a = \<\n_c e_b, e_a\>, \;\; \G_{ac}^b = - \G_{bc}^a,
\qquad \Om^a_b = - \Om^b_a = \tfrac 12 \sum\nolimits_{c,d} R_{abcd} \theta^c \wedge \theta^d.
\end{equation*}
We use the same notation with the bar for the corresponding objects relative to the metric $g$ on $N$; the convention for the index ranges is $0 \le a,b,c, \dots \le n$ and $1 \le i,j,k, \dots \le n$.

We have $d\theta^0=0, \; d\theta^i=e^{up_i}(p_i \theta^0 + u \sum\nolimits_j e_j(p_i) \theta^j) \wedge \overline{\theta}^i - e^{up_i}\sum\nolimits_j \overline{\psi}^i_j \wedge \overline{\theta}^j$, so
\begin{equation*}
    \sum\nolimits_{b,c} \G^0_{bc} \theta^c \wedge \theta^b=0, \quad
    \sum\nolimits_{b,c} \G^i_{bc} \theta^c \wedge \theta^b=-p_i \theta^0 \wedge \theta^i- u e^{up_i}\sum\nolimits_{j}\overline{e_j}(p_i) \overline{\theta}^j \wedge \overline{\theta}^i  + e^{up_i}\sum\nolimits_{j,k}\overline{\G}^i_{jk} \overline{\theta}^k \wedge \overline{\theta}^j.
\end{equation*}
It follows that
\begin{equation}\label{eq:gij0}
    \G^i_{00}=\G^0_{i0}=0, \quad \G^0_{ii}=p_i, \quad \G^i_{j0}=\G^0_{ij}=0, \quad \text{for } i \ne j,
\end{equation}
and that $\sum\nolimits_{j,k}(e^{u(p_j+p_k)}\G^i_{jk} -e^{up_i}(u \overline{e_j}(p_i) \K_{ik} + \overline{\G}^i_{jk}) )\overline{\theta}^j \wedge \overline{\theta}^k= 0$. By the cyclic permutation of $i,j,k$ we obtain 
\begin{equation}\label{eq:gijk}
\begin{split}
    \G^i_{jk}=& \tfrac12 e^{u(p_k-p_i-p_j)}(\overline{\G}^k_{ji}-\overline{\G}^k_{ij}+u (\K_{ki}\overline{e_j}-\K_{kj}\overline{e_i})(p_k))\\
     -&\tfrac12 e^{u(p_i-p_j-p_k)}(\overline{\G}^i_{kj}-\overline{\G}^i_{jk}+u (\K_{ij}\overline{e_k}-\K_{ik}\overline{e_j})(p_i))\\
     -&\tfrac12 e^{u(p_j-p_k-p_i)}(\overline{\G}^j_{ik}-\overline{\G}^j_{ki}+u (\K_{jk}\overline{e_i}-\K_{ji}\overline{e_k})(p_j)).
\end{split}
\end{equation}
By \eqref{eq:structure} the Ricci tensor $\ric$ of $(M',g^D)$ is given by
\begin{equation}\label{eq:ricci}
    \ric_{ab}=\sum\nolimits_c R_{cacb}=\sum\nolimits_c \big(e_c(\G^c_{ab})-e_b(\G^c_{ac})+ \sum\nolimits_d(\G^c_{dc}\G^d_{ab}-\G^c_{ad}\G^d_{bc})\big). 
\end{equation}
Then by \eqref{eq:gij0}
\begin{equation}\label{eq:ric00}
    \ric_{00}=-\sum\nolimits_i p_i^2=-\Tr (D^2),
\end{equation}
and by (\ref{eq:gij0}, \ref{eq:gijk}),
\begin{equation}\label{eq:ric0i}
\begin{split}
    \ric_{0i}&=e^{-up_i}\Big(\overline{e}_i(\Tr D -p_i)+\sum\nolimits_j\overline{\G}^i_{jj}(p_i-p_j)+u(\overline{e}_i(\tfrac12\Tr (D^2))-p_i\overline{e}_i(\Tr D))\Big)\\
    &=e^{-up_i}\Big(\overline{e}_i(\Tr D)-\overline{\<\overline{e}_i, \sum\nolimits_j(\overline{\n}_{\overline{e}_j}D)\overline{e}_j\>} +u(\overline{e}_i(\tfrac12\Tr(D^2))-p_i\overline{e}_i(\Tr D))\Big).
\end{split}
\end{equation}
Note that the connection forms of the metric $g^u$ on $N$ relative to the frame $e_i$ are the same as those of $(M,g^D)$ and are given by \eqref{eq:gijk}, so from \eqref{eq:ricci} (or from the Gauss equation) we obtain
\begin{equation}\label{eq:ricij}
    \ric_{ij}=\ric^u_{ij}-\K_{ij} p_i \Tr D,
\end{equation}
where $\ric^u$ is the Ricci tensor of $(N,g^u)$. 

\subsection{Proof of Theorem~\ref{th:Dconst}}
\label{ss:pfofDconst}
Suppose $(M, g^D)$ is Einstein. Let $U$ be an open, connected domain of $(N, g)$ on which the eigenvalues $p_i$ of $D$ have constant multiplicities and on which the rank of $D$ is constant. Let $M':=\br \times U \subset M$. Then by \eqref{eq:ric00} the Einstein constant is $-\Tr (D^2)$, so from \eqref{eq:ric0i} and \eqref{eq:ricij} at the points of $(M', g^D)$ we obtain:
\begin{align}
    &p_i \overline{e}_i (\Tr D)=0, \quad \text{for all } i=1, \dots, n, \label{eq:dTr} \\
    &\overline{\div} D = \overline{\grad} (\Tr D), \label{eq:divgrad} \\
    &\Ric^u= (\Tr D) \, D - \Tr (D^2) \, \id, \label{eq:RicuD}
\end{align}
where $\overline{\grad}$ is the gradient relative to $g$.

The right-hand side of \eqref{eq:RicuD} is independent of $u$, while the left-hand side, by \eqref{eq:ricci} and \eqref{eq:gijk}, is the sum of expressions of the form $f_\a e^{ul_\a}, \, f_\a ue^{ul_\a}$, and $f_\a u^2e^{ul_\a}$, where $l_\a$ are linear combinations of the $p_i$'s with integer coefficients, and $f_\a$ are functions on $N$. For any sum $F$ of such expressions, we denote by $[u^2 \exp]F$ the sum of all terms of $F$ of the form $f_\a u^2e^{ul_\a}$. Then \eqref{eq:RicuD} gives $[u^2 \exp]\ric^u_{ij}=0$. To compute $\ric^u_{ij}$ we use \eqref{eq:ricci} replacing $a,b,c,d$ with $i,j,k,l$.
Then from \eqref{eq:gijk} we obtain
\begin{equation*}
\begin{split}
    0&=[u^2 \exp]\ric^u_{ij}=[u^2 \exp](e^{-up_k}\overline{e}_k(\G^k_{ij})-e^{-up_j}\overline{e}_j(\G^k_{ik}) -\G^k_{lj}\G^l_{ik}+\G^k_{lk}\G^l_{ij}+\G^k_{il}\G^l_{kj}-\G^k_{il}\G^l_{jk})\\
    &=u^2 e^{-u(p_i+p_j)}\Bigl(\overline{e}_i(\Tr D)\overline{e}_j(p_i)+\overline{e}_j(\Tr D)\overline{e}_i(p_j)-2\overline{e}_j(p_i)\overline{e}_i(p_j) -\sum\nolimits_k\overline{e}_j(p_k)\overline{e}_i(p_k)\Bigr)\\ &+\K_{ij}u^2\sum\nolimits_k e^{-2up_k}\overline{e}_k(p_i)\overline{e}_k(2p_k- \Tr D).
\end{split}
\end{equation*}
In particular, taking $i=j$ and summing up by $i=1, \dots , n$ we get 
\begin{multline*}
\sum\nolimits_i\bigl(e^{-2up_i}\bigl(2\overline{e}_i(\Tr D)\overline{e}_i(p_i)-2(\overline{e}_i(p_i))^2 -\sum\nolimits_k(\overline{e}_i(p_k))^2\bigr)\bigr) \\ +\sum\nolimits_k e^{-2up_k}\overline{e}_k(\Tr D)\overline{e}_k(2p_k- \Tr D)=0.
\end{multline*}
On the domain $U \subset N$, the rank $r: = \rk D$ is constant, and by relabelling we can assume that $p_{r+1}=\dots=p_{n}=0$ and that all the functions $p_1, \dots, p_r$ are nonzero on $U$. Then from \eqref{eq:dTr} we get $\overline{e}_i(p_i)\overline{e}_i(\Tr D)=0$, for all $i=1, \dots, n$, and so the above equation gives
\begin{equation*}
-\sum\nolimits_i e^{-2up_i}\bigl(2(\overline{e}_i(p_i))^2+\sum\nolimits_k(\overline{e}_i(p_k))^2\bigr)
    -\sum\nolimits_k e^{-2up_k}(\overline{e}_k(\Tr D))^2=0.
\end{equation*}
Thus $\overline{e}_i(p_k)=0$, for all $i,k = 1, \dots, n$, at all the points of $U$. It follows that the coefficients of the characteristic polynomial of $D$ are locally constant on every open, connected domain $U \subset N$ where the eigenvalues of $D$ have constant multiplicities. As the union of such domains is dense in $N$, the eigenvalues of $D$ are constant on the whole manifold $N$. Then (\ref{eq:divgrad}, \ref{eq:RicuD}) imply \eqref{eq:ein1div}, as required.

The converse easily follows from \eqref{eq:ric0i} and \eqref{eq:ricij}.

\subsection{Ricci and scalar curvature of the $D$-deformation with constant eigenvalues}
\label{ss:ricuEin}
In this subsection, we compute the Ricci tensor and the scalar curvature of the $D$-deformation \emph{assuming the eigenvalues of $D$ to be constant}. Note that by Theorem~\ref{th:Dconst}\eqref{it:Dconst1}, this condition must be always satisfied if the $D$-extension is Einstein.

From \eqref{eq:gijk}, the connection components of $g^u$ are given by
\begin{equation}\label{eq:gijkc}
\begin{split}
    \G^i_{jk}&= \tfrac12 e^{u(p_k-p_i-p_j)}(\overline{\G}^k_{ji}-\overline{\G}^k_{ij})
     -\tfrac12 e^{u(p_i-p_j-p_k)}(\overline{\G}^i_{kj}-\overline{\G}^i_{jk})
     -\tfrac12 e^{u(p_j-p_k-p_i)}(\overline{\G}^j_{ik}-\overline{\G}^j_{ki})\\
     &=\tfrac12 \overline{\G}^i_{jk}(e^{u(p_i-p_j-p_k)}+e^{u(p_j-p_k-p_i)})\\
     &+\tfrac12\overline{\G}^k_{ji}(e^{u(p_k-p_i-p_j)}-e^{u(p_j-p_k-p_i)})-\tfrac12 \overline{\G}^k_{ij}(e^{u(p_k-p_i-p_j)}-e^{u(p_i-p_j-p_k)}).
\end{split}
\end{equation}

Introduce the functions $\mu_{ij|k}=\overline{\<[\overline{e}_i,\overline{e}_j], \overline{e}_k\>}$. Then from \eqref{eq:gijkc} we have
\begin{equation}\label{eq:gijklie}
\begin{split}
    \mu_{ij|k}&=\overline{\G}_{ji}^k-\overline{\G}_{ij}^k,\\
    \G^i_{jk}&= \tfrac12 e^{u(p_k-p_i-p_j)}\mu_{ij|k}-\tfrac12 e^{u(p_i-p_j-p_k)}\mu_{jk|i}-\tfrac12 e^{u(p_j-p_k-p_i)}\mu_{ki|j}.
\end{split}
\end{equation}
Substituting into \eqref{eq:ricci} we obtain:
\begin{equation}\label{eq:Ricuij}
\begin{split}
\<\Ric^u e_i, e_j\>_u= &-\tfrac12e^{-u(p_i+p_j)}\bigl(\overline{e}_j(\sum\nolimits_{k}\mu_{ki|k}) + \overline{e}_i(\sum\nolimits_{k}\mu_{kj|k}) +\sum\nolimits_{k,l}\mu_{jk|l}\mu_{il|k}\bigr) \\
&+\tfrac12e^{u(p_i-p_j)}\bigl(\sum\nolimits_{k}e^{-2up_k}\overline{e}_k(\mu_{kj|i})+\sum\nolimits_{k,l}e^{-2up_l}\mu_{lj|i}\mu_{kl|k}\bigr) \\
&+\tfrac12e^{u(p_j-p_i)}\bigl(\sum\nolimits_{k}e^{-2up_k}\overline{e}_k(\mu_{ki|j})+\sum\nolimits_{k,l}e^{-2up_l}\mu_{li|j}\mu_{kl|k}\bigr) \\
&+\tfrac14e^{u(p_i+p_j)}\sum\nolimits_{k,l}e^{-2u(p_l+p_k)}\mu_{kl|i}\mu_{kl|j} \\
&-\tfrac12e^{-u(p_i+p_j)}\sum\nolimits_{k,l}e^{2u(p_l-p_k)}\mu_{ik|l}\mu_{jk|l}.
\end{split}
\end{equation}
In particular, we get
\begin{equation}\label{eq:Ricuii}
\begin{split}
\<\Ric^u e_i, e_i\>_u= &-\tfrac12e^{-2up_i}\bigl(2\overline{e}_i(\sum\nolimits_{k}\mu_{ki|k}) +\sum\nolimits_{k,l}\mu_{ik|l}\mu_{il|k}\bigr) \\
&+\sum\nolimits_{k}e^{-2up_k}\bigl(\overline{e}_k(\mu_{ki|i})+\mu_{ki|i}\sum\nolimits_{l}\mu_{lk|l}\bigr) \\
&+\tfrac14\sum\nolimits_{k,l}e^{2u(p_i-p_l-p_k)}(\mu_{kl|i})^{2} -\tfrac12\sum\nolimits_{k,l}e^{2u(p_l-p_k-p_i)}(\mu_{ik|l})^{2},
\end{split}
\end{equation}
and for the scalar curvature of the $D$-deformation,
\begin{equation} \label{eq:scalu}
\begin{split}
\scal^u =&  \sum\nolimits_{k}e^{-2up_k}\Bigl(2\overline{e}_k\Bigl(\sum\nolimits_{i}\mu_{ki|i}\Bigr)-\Bigl(\sum\nolimits_{i}\mu_{ki|i}\Bigr)^2-\tfrac12 \sum\nolimits_{i,l}\mu_{ki|l}\mu_{kl|i}\Bigr) \\
&-\tfrac14\sum\nolimits_{i,k,l}e^{2u(p_i-p_l-p_k)}(\mu_{kl|i})^{2}.
\end{split}
\end{equation}

\smallskip

If we additionally assume the $D$-extension to be Einstein, then by \eqref{eq:ein1Ru} we obtain
\begin{align}\label{eq:ein1Rudetailes}
  \<\Ric^u e_i, e_j\>_u & = ((\Tr D) p_i - \Tr (D^2)) \K_{ij}, \\
  \scal^u & = (\Tr D)^2  - n \Tr (D^2), \label{eq:scaludetailes}
\end{align}
where the left-hand sides are given by \eqref{eq:Ricuij} and \eqref{eq:scalu} respectively. We also note that \eqref{eq:ein1div} can be explicitly written in the equivalent form
\begin{equation}\label{eq:mijj}
    \sum\nolimits_j\overline{\G}^i_{jj}(p_i-p_j)=\sum\nolimits_j \mu_{ij|j}(p_i-p_j)=0.
\end{equation}

Summarising the above we can express the conditions of Theorem~\ref{th:Dconst} explicitly as follows.

\begin{corollary} \label{cor:Einext}
The extension $(M,g^D)$ is Einstein if and only if $D$ has constant eigenvalues and equations \eqref{eq:ein1Rudetailes} \emph{(}with $\Ric^u$ given by \eqref{eq:Ricuij}\emph{)} and \eqref{eq:mijj} are satisfied.
\end{corollary}

\section{The eigenvalue type of $D$. Proof of Theorem~\ref{th:nonzero}}
\label{s:eigen}

By Theorem~\ref{th:Dconst}, the eigenvalues $p_i$ of $D$ are constant. In this section, we show that there are strong algebraic restrictions on $p_i$.

We have the following lemma (note that assertion~\eqref{it:sc1} is well-known, as in that case $g^D$ is a warped product).
\begin{lemma}\label{l:scalar}
Suppose that $(M,g^D)$ is Einstein. Then
\begin{enumerate}[label=\emph{(\roman*)},ref=\roman*]
  \item \label{it:sc1}
  $D$ is scalar if and only if $N$ is Ricci flat.
  \item \label{it:sc2}
  In all the other cases, there exist $i, j, k$ with $i \ne j$ such that $p_k=p_i+p_j$.
  \item \label{it:sc3}
  If $p_i+p_j-p_k \notin \{0, p_1, \dots, p_n\}$, then $\mu_{ij|k}=0$.
\end{enumerate}
\end{lemma}
\begin{proof}
\eqref{it:sc1} Suppose that $D$ is scalar. Then from \eqref{eq:ein1Ru} $\Ric^u=0$. In particular, $\Ric^0=0$, as required.

Conversely, if $\Ric^0=0$, then by \eqref{eq:ein1Ru} we have $(\Tr D) \, D =  \Tr (D^2) \, \id$. Taking the traces of both sides we obtain $(\Tr D)^2 = n\Tr (D^2)$, which is only possible when $D$ is scalar, by the Cauchy-Schwarz inequality.

\eqref{it:sc2} Suppose such a triple $i, j, k$ does not exist. In particular, this means that all the $p_i$ are nonzero. Then the only possible way for the scalar curvature $\scal^u$ given by \eqref{eq:scalu} to be constant is when it is identically zero. By \eqref{eq:scaludetailes} this implies that $(\Tr D)^2 =  n\Tr (D^2)$, so that $D$ is scalar.

\eqref{it:sc3} Let $\mathcal{S}$ be the set of all triples $(i,j,k)$, with $i \ne j$, satisfying the assumption of the assertion (note that for $i=j$, the claim is trivial). Then $k \ne i, j$ (as otherwise $p_i+p_j-p_k \in \{p_i,p_j\}$). Then from \eqref{eq:scalu}, $\scal^u =   \sum\nolimits_{k}e^{-2up_k}(\dots) + (\dots) -\tfrac14\sum\nolimits_{(k,l,i) \in \mathcal{S}}e^{2u(p_i-p_l-p_k)}(\mu_{kl|i})^{2}$, where $(\dots)$ are some expressions not involving $u$. As no terms in the last sum have zero exponents or the same exponents as the terms in the first sum (by the definition of $\mathcal{S}$), and as $\scal^u$ is a constant, we get $\sum\nolimits_{(k,l,i) \in \mathcal{S}}e^{2u(p_i-p_l-p_k)}(\mu_{kl|i})^{2}=0$, so $\mu_{kl|i}=0$, for all $(k,l,i) \in \mathcal{S}$.
\end{proof}

\begin{remark}\label{rem:survive}
It follows from Lemma~\ref{l:scalar}\eqref{it:sc3} and from \eqref{eq:Ricuii} and \eqref{eq:ein1Rudetailes} that in the diagonal components $(\Ric^u)^i_i=\<\Ric^u e_i, e_i\>_u$ of the Ricci tensor of $g^u$, the only non-vanishing terms are those with $e^{-2uq}$, where $q \in P:=\{0, p_1, \dots, p_n\}$. Let $P=\{0, q_1, \dots, q_m\}$ (without repetitions). Collecting the similar terms in \eqref{eq:Ricuii}, we obtain
\begin{equation}\label{eq:collect}
    (\Ric^u)^i_i=\sum\nolimits_{a=1}^m e^{-2uq_a} r_{ia} +r_{i0},
\end{equation}
where the expressions $r_{ia}, r_{i0}$ do not depend on $u$ (but only on $(N, g)$), and so for all $i=1, \dots, n$, \eqref{eq:ein1Rudetailes} implies
\begin{equation}\label{eq:equate}
    r_{ia}=0, \;\text{ for all } a=1, \dots, m, \quad \text{and } r_{i0}=\Big(\sum\nolimits_j p_j\Big) p_i - \Big(\sum\nolimits_j p_j^2\Big).
\end{equation}
\end{remark}

\begin{proof}[Proof of Theorem~\ref{th:nonzero}]
\eqref{it:nonzero1} Suppose that all the $p_i$ are nonzero. Then from \eqref{eq:Ricuii} and \eqref{eq:collect}, the terms $r_{i0}$ of the diagonal elements $(\Ric^u)_i^i=\<\Ric^u e_i, e_i\>_u$ of $\Ric^u$ are given by
\begin{equation*}
r_{i0}=\tfrac14\sum\nolimits_{k,l: p_i-p_l-p_k=0}(\mu_{kl|i})^{2} -\tfrac12\sum\nolimits_{k,l: p_l-p_k-p_i=0}(\mu_{ik|l})^{2},
\end{equation*}
Note that, as $p_i \ne0$, all the three subscripts $i,k,l$ in each of the above summations are pairwise non-equal. It follows that
\begin{equation*}
\begin{split}
r_{i0}&=\tfrac14\sum\nolimits_{k,l: f_k+f_l-f_i \in (F \cap {\mathbf{p}}^\perp)}(\mu_{kl|i})^{2} -\tfrac12\sum\nolimits_{k,l: f_i+f_k-f_l \in (F \cap {\mathbf{p}}^\perp)}(\mu_{ik|l})^{2}\\
&=\tfrac12\sum\nolimits_{a: v_a=f_k+f_l-f_i \in (F \cap {\mathbf{p}}^\perp)}(\mu_{kl|i})^{2} -\tfrac12\sum\nolimits_{a: v_a=f_i+f_k-f_l \in (F \cap {\mathbf{p}}^\perp)}(\mu_{ik|l})^{2},
\end{split}
\end{equation*}
where $v_a$ is some labeling of the elements of the set $F \cap {\mathbf{p}}^\perp$. Denote $\mu_a^2=(\mu_{ij|k})^2$ for $v_a=f_i+f_j-f_k \in (F \cap {\mathbf{p}}^\perp)$. Then we obtain
\begin{equation*}
r_{i0}=\tfrac12\sum\nolimits_{a: v_a \in (F \cap {\mathbf{p}}^\perp), \<v_a,f_i\>=-1}\mu_a^2 -\tfrac12\sum\nolimits_{a: v_a \in (F \cap {\mathbf{p}}^\perp), \<v_a,f_i\>=1}\mu_a^2=-\tfrac12\sum\nolimits_{a: v_a \in (F \cap {\mathbf{p}}^\perp)}\<v_a,f_i\>\mu_a^2.
\end{equation*}
By the second equation of \eqref{eq:equate}, the latter expression equals $\<{\mathbf{p}},\mathbf{1}_n\>p_i-\|{\mathbf{p}}\|^2$, so we obtain
\begin{equation} \label{eq:Fpperp}
\sum\nolimits_{a: v_a \in (F \cap {\mathbf{p}}^\perp)} (\tfrac12\mu_a^2) v_a =\|{\mathbf{p}}\|^2\mathbf{1}_n-\<{\mathbf{p}},\mathbf{1}_n\>{\mathbf{p}}.
\end{equation}
The right-hand side is the projection of $\mathbf{1}_n$ to ${\mathbf{p}}^\perp$ multiplied by $\|{\mathbf{p}}\|^2$, and the claim follows.

\eqref{it:nonzero2} From \eqref{eq:Fpperp} we obtain that $Vc=\|{\mathbf{p}}\|^2\mathbf{1}_n-\<{\mathbf{p}},\mathbf{1}_n\>{\mathbf{p}}$, for some vector $c \in \br^m$ (note that its components are not necessarily nonnegative). By construction, $V^t \mathbf{1}_n = \mathbf{1}_m, \; \rk V = m$ and $V^t{\mathbf{p}}=0$ (in particular, $n > m$, as otherwise ${\mathbf{p}}=0$). Then $V^tVc=\|{\mathbf{p}}\|^2\mathbf{1}_m$, so $c=\|{\mathbf{p}}\|^2 (V^tV)^{-1}\mathbf{1}_m$. Therefore $\<{\mathbf{p}},\mathbf{1}_n\>{\mathbf{p}}=\|{\mathbf{p}}\|^2(\mathbf{1}_n-V(V^tV)^{-1}\mathbf{1}_m)$, and equation \eqref{eq:eigenstr} follows, as $\<{\mathbf{p}},\mathbf{1}_n\>= \Tr D$ is assumed to be nonzero.

As all the components of the vector on the right-hand side of \eqref{eq:eigenstr} are rational, all the $p_i$'s, up to scaling, are integer.

The last claim of the assertion follows from the fact that for every $n$, there is a finite number of possible matrices $V$.
\end{proof}

\begin{remark}\label{rem:34}
The conditions imposed by Theorem~\ref{th:nonzero}\eqref{it:nonzero1} on the eigenvalue type of $D$ are quite restrictive (although somewhat implicit). For example, it follows that if $n=3$ and $\det D \ne 0$, then $D= \diag(1,1,2)$ and $D= \id$ are the only possible eigenvalue types, up to scaling. Indeed, there can be no more than one relation of the form $p_i+p_j-p_k=0$, with $i \ne j$, between $p_1,p_2,p_3$. It follows that $F \cap {\mathbf{p}}^\perp$ is either empty (then $D$ is scalar by Lemma~\ref{l:scalar}\eqref{it:sc2}), or consists of a single element $f_1+f_2-f_3$, up to relabelling. In the latter case, the matrix $V$ is $3 \times 1$, $V=(1,1,-1)^t$, and by \eqref{eq:eigenstr}, the vector $\mathbf{p}$ is a multiple of $(1,1,2)^t$. If $n=4$ and $\det D \ne 0$, then considering all the possibilities for the matrix $V$ we obtain that all the eigenvalue types, up to scaling, are
\begin{gather*}
(1,1,1,1)^t,\; (2,2,3,4)^t,\; (3,4,4,7)^t,\; (1,2,3,4)^t,\; (1,1,2,2)^t,\\ (1,1,1,2)^t,\; (1,1,2,3)^t, \; (-1,1,1,2)^t,\; (-1,1,2,3)^t.
\end{gather*}
\end{remark}

\begin{remark}\label{rem:drop}
In general, the condition $\det D \ne 0$ in Theorem~\ref{th:nonzero} cannot be dropped, as shows the analysis of the case $n=3$ in Section~\ref{s:dim3} (for example, in the last row of Table~\ref{t:d3} in Theorem~\ref{th:dim3}, $p$ can be any real number). One might ask however, if the assumption $\Tr D \ne 0$ in Theorem~\ref{th:nonzero}\eqref{it:nonzero2} can be removed. It follows from the proof that the condition $\Tr D = 0$ is equivalent to $V(V^tV)^{-1}\mathbf{1}_m=\mathbf{1}_n$ (which geometrically means that the manifold $(N,g)$ is by itself Einstein). The following example shows the necessity of this assumption, at least at the algebraic level. Let $n=6$ and let $\mathbf{p}=(-3,-2,-1,1,2,3)^t$. Then $F \cap {\mathbf{p}}^\perp = \{v_{142}, v_{153}, v_{231}, v_{243}, v_{264}, v_{354},v_{365},v_{456}\}$, where $v_{ijk}=f_i+f_j-f_k$. The projection of $\mathbf{1}_6$ to ${\mathbf{p}}^\perp$ is $\mathbf{1}_6$ which belongs to the convex cone hull of the set $F \cap {\mathbf{p}}^\perp$ as $2 \cdot \mathbf{1}_6 = 3v_{142}+v_{153}+2v_{231}+v_{243}+2v_{264}+v_{354}+v_{365}+v_{456}$, but \eqref{eq:eigenstr} is not satisfied (taking the inner product of both sides with $\mathbf{p}$ we obtain $\|\mathbf{p}\|^2=0$).
\end{remark}

\section{Einstein $D$-extension of the eigenvalue type $\mathbf{p} = (\lambda, \cdots, \lambda, \nu)^t$}
\label{s:Heis}


\subsection{Three cases}
\label{s:threecases}
In the previous section, we considered the conditions which the Einstein property of $(M,g^D)$ imposes on $D$. Clearly, there are also some conditions on $(N,g)$. For example, the Ricci eigenvalues of $(N,g)$ are constant by \eqref{eq:ein1Ru}, and the scalar curvature is negative, unless $(N,g)$ is Ricci flat, by Lemma~\ref{l:scalar}\eqref{it:sc1}.

In this section, we give a complete characterisation of $(N,g)$ and of $(M,g^D)$ in the case when $D$ has the ``next simplest" eigenvalue type after being a scalar operator, namely when one of the eigenvalues of $D$ has multiplicity $n-1$. As we know from the Introduction, there are only three possibilities, up to scaling; the Ricci operator $\Ric^u$ is given by \eqref{eq:ein1Ru}:
\begin{enumerate}[label=(\alph*),ref={\bf{\alph*}}]
  \item \label{it:01}
  $\mathbf{p}=(0, \dots, 0, 1),\quad \Ric^u = \diag (-1, \dots, -1, 0)$,

  \item \label{it:10}
  $\mathbf{p}=(1, \dots, 1, 0),\quad \Ric^u = \diag (0, \dots, 0, 1-n)$,

  \item  \label{it:12}
  $\mathbf{p}=(1, \dots, 1, 2),\quad \Ric^u = \diag (-2, \dots, -2, n-1)$.
\end{enumerate}
We consider them in Theorem~\ref{th:01}, Theorem~\ref{th:10}, and Theorem~\ref{th:12} respectively. We start with the following lemma valid in all three cases. Denote $S_i=\sum_k \mu_{ki|k}$.

\begin{lemma} \label{l:multn-1}
Let $(M,g^D)$ be the $D$-extension of $(N,g)$ with the spectral vector $\mathbf{p} = (\lambda, \cdots, \lambda, \nu)^t$, $\lambda \ne \nu$. Then
\begin{enumerate}[label=\emph{(\roman*)},ref=\roman*]
  \item \label{it:multn-1Smu}
  $\div D = 0$ if and only if $S_n=0$ and $\mu_{in|n}=0$, for all $i < n$.

  \item \label{it:multn-1sym}
  Locally there exists a frame $\overline{e}_1, \dots, \overline{e}_{n-1}$ for the $\la$-eigendistribution of $D$ such that $\mu_{ni|j}=\mu_{nj|i}$, for all $i,j = 1, \dots, n-1$. Such a frame can be chosen arbitrarily on a hypersurface transversal to $\overline{e}_n$.
\end{enumerate}
\end{lemma}
\begin{proof}
\eqref{it:multn-1Smu} follows from \eqref{eq:mijj}.

\eqref{it:multn-1sym} Let $W = (w_{ij})$ be an $(n-1)\times (n-1)$ orthogonal matrix whose entries are smooth functions on $N$, and define $\overline{e}'_i=\sum_k w_{ik}\overline{e}_k$. Relative to the orthonormal frame $\overline{e}'_1, \dots , \overline{e}'_{n-1},\overline{e}_n$, we have $\mu'_{ni|j} = \<[\overline{e}_n,\overline{e}'_i],\overline{e}'_j\> = \sum_k \overline{e}_n(w_{ik}) w_{jk}+\sum_{k,s} w_{ik} w_{js} \mu_{nk|s}$. Let $K, K'$ be $(n-1)\times (n-1)$ skew-symmetric matrices defined by $K_{ij}=\mu_{ni|j}-\mu_{nj|i}$ and $K'_{ij}=\mu'_{ni|j}-\mu'_{nj|i}$. Then $K'=WKW^t + \overline{e}_n(W)W^t - W\overline{e}_n(W^t)$ which is equivalent to $W^tK'W=K+2W^t\overline{e}_n(W)$. Solving the equation $\overline{e}_n(W)=-\frac12 WK$ along the integral curves of $\overline{e}_n$, with the initial condition $W=\id$ on a hypersurface transversal to $\overline{e}_n$, we get $K'=0$, as required.
\end{proof}

For the rest of this section, we assume the frame $\overline{e}_i$ to be chosen as in Lemma~\ref{l:multn-1}\eqref{it:multn-1sym}.

\subsection{Proof of Theorem~\ref{th:01}} \label{ss:01}
From Lemma~\ref{l:scalar}\eqref{it:sc3} we get $\mu_{ij|n}=0$, for all $i,j < n$. Using this fact and Lemma~\ref{l:multn-1} we find $(\Ric^u)_n^n= -e^{-2u} \sum\nolimits_{k,l < n}(\mu_{nk|l})^2$ from \eqref{eq:Ricuii}. But $(\Ric^u)_n^n = 0$ by \eqref{eq:ein1Ru}, so $\mu_{nk|l}=0$, for all $k,l$. Therefore $\mu_{ij|k}=0$ whenever at least one of the subscripts equals $n$. It follows that the vector field $\overline{e}_n$ is parallel, and so $(N,g)$ is locally isometric to the Riemannian product of the real line and an $(n-1)$-dimensional manifold $N'$. Furthermore, for $i,j<n$ we get $(\Ric^u)_i^j = -\K_{ij}$ by \eqref{eq:ein1Ru}. It follows that $N'$ is Einstein, with the Einstein constant $-1$. Then the metric $g$ on $N$ is given by $\overline{ds^2}=(dx^n)^2 + ds'^2$, where $ds'{}^2$ is an Einstein metric on $N'$, and hence the metric $g^D$ on $M$ is given by $\overline{ds^2}=du^2+e^{2u}(dx^n)^2 + ds'^2$. Thus $(M,g^D)$ is locally isometric to the Riemannian product of the hyperbolic plane of curvature $-1$ and the Einstein manifold $N'$ with the Einstein constant $-1$.

\begin{remark} \label{rem:convth01}
Note that the converse to Theorem~\ref{th:01} is easily verified: starting with the Riemannian product of the real line $\br$ and an Einstein manifold $N'$ of dimension $n-1$ with the Einstein constant $-1$ and extending it by the endomorphisms $D$ with the spectral vector $\mathbf{p}=(0, \dots, 0, 1)^t$ whose kernel is $TN'$ we get an Einstein manifold $(M,g^D)$ isometric to the Riemannian product of the hyperbolic plane of curvature $-1$ and $N'$.
\end{remark}

\subsection{Case~\eqref{it:10}} \label{ss:10}
In the case $\mathbf{p}=(1, \dots,1, 0)^t$, we have the following theorem.

\begin{theorem}\label{th:10}
Let $(M,g^D)$ be the $D$-extension of $(N,g)$ with the spectral vector $\mathbf{p}=(1, \dots, 1, 0)^t$. If $(M,g^D)$ is Einstein, then the $1$-eigendistribution of $D$ is integrable and the manifold $(N,g)$ is locally an extension of a Ricci flat manifold $(N',g')$ of dimension $n-1$ by a field of symmetric endomorphism $D'$ with constant eigenvalues, such that $\Tr D' = 0$, $\Tr D'^2=n-1$, and $\div' D'=0$. The Einstein metric $g^D$ on $M$ is locally given by $ds^2=du^2+dt^2+\sum_{i=1}^{n-1} e^{2u+2tq_i} (\theta'^i)^2$, where $q_i$ are the eigenvalues of $D'$ and $\{\theta'^i\}$ is the coframe dual to an orthonormal frame of eigenvectors of $D'$ on $N'$. 
\end{theorem}
\begin{proof}
From Lemma~\ref{l:multn-1}\eqref{it:multn-1Smu} and Lemma~\ref{l:scalar}\eqref{it:sc3} we obtain that $S_n=0$ and that $\mu_{ij|n}=0$, for all $i,j$. It follows that the vector field $\overline{e}_n$ is geodesic and that its orthogonal distribution is integrable. Therefore we can define a function $x^n$ locally on $N$ in such a way that $\overline{\theta}^n=dx^n$. Without loss of generality, assume that $N'$ is the level hypersurface of $N$ defined by the equation $x^n=0$. Denote $g'$ the induced metric on $N'$.

From \eqref{eq:ein1Rudetailes} we get $(\Ric^u)_i^j=0$ for $i,j < n$. Then by \eqref{eq:Ricuij}, the sum of the terms of $(\Ric^u)_i^j$ which do not depend on $u$ gives $\overline{e}_n(\mu_{ni|j})=0$, for all $i,j < n$. We can now specify the frame $\overline{e}_i, \; i<n$, further. Note that $\mu_{ni|j}= \<[\overline{e}_n, \overline{e}_i],\overline{e}_j\> = \<\overline{\nabla}_{\overline{e}_n} \overline{e}_i, \overline{e}_j\> - \<\overline{\nabla}_{\overline{e}_i}\overline{e}_n, \overline{e}_j\> = \<\overline{\nabla}_{\overline{e}_n} \overline{e}_i, \overline{e}_j\> + \<\overline{\nabla}_{\overline{e}_i}\overline{e}_j, \overline{e}_n\>$. On the hypersurface $N'$, we have $\<\overline{\nabla}_{\overline{e}_i} \overline{e}_j, \overline{e}_n\>=h(\overline{e}_i, \overline{e}_j)$, where $h$ is the second fundamental form of $N'$. As $\mu_{ni|j}=\mu_{nj|i}$ we obtain that on $N'$, the expression $\<\overline{\nabla}_{\overline{e}_n} \overline{e}_i, \overline{e}_j\>$ is symmetric relative to $i,j$, hence $\<\overline{\nabla}_{\overline{e}_n} \overline{e}_i, \overline{e}_j\>=0$, and so $\mu_{ni|j}=h(\overline{e}_i, \overline{e}_j)$ on $N'$. We can now choose the frame $\overline{e}_i, \; i<n$, on $N'$ consisting of orthonormal eigenvectors of the second fundamental form $h$ and then extend it to $N$ as in Lemma~\ref{l:multn-1}\eqref{it:multn-1sym}. Then the symmetric matrix $(\mu_{ni|j})$ restricted to $N'$ is diagonal, and from the fact that  $\overline{e}_n(\mu_{ni|j})=0$, for all $i,j < n$, we obtain that $(\mu_{ni|j})$ is diagonal locally on $N$, so that $\mu_{ni|j}=q_i \K_{ij}$ and $\overline{e}_n(q_i)=0$. Then the vector fields $e'_i=e^{-x^nq_i} \overline{e}_i, \; i<n$, are Lie parallel along the integral curves of $\overline{e}_n=\partial_n$, as also are their dual one-forms ${\theta'}^i= e^{x^nq_i} \overline{\theta}^i$. It follows that the metric $g$ on $N$ is locally given by $\overline{ds^2}=\sum\nolimits_{i=1}^n (\overline{\theta}^i)^2 = (dx^n)^2+\sum\nolimits_{i<n} e^{2x^nq_i}({\theta'}^i)^2$. Thus the metric $(M,g^D)$ has the required form and moreover, by \eqref{eq:metricM}, $(N,g)$ is locally an extension of $(N',g')$ by the field of the symmetric endomorphism $D'$ defined by $D'\overline{e}_i=q_i \overline{e}_i$ at the points of $N'$.

We also have $\Tr D' = \sum_i \mu_{ni|i}=-S_n=0$ by Lemma~\ref{l:multn-1}\eqref{it:multn-1Smu}. Furthermore, from \eqref{eq:ein1Rudetailes} we get $(\Ric^u)_n^n=1-n$, which by \eqref{eq:Ricuii} gives $\Tr D'^2 = \sum_i \mu_{ni|i}^2 = n-1$. Moreover, as $(\Ric^u)_i^j=0$ for $i,j < n$, we have $(\Ric^0)_i^j=0$ (note that $g^0=g$), and so \eqref{eq:ricij} applied to the extension of $(N',g')$ by $D'$ implies that $(N',g')$ is Ricci flat. We also have $(\Ric^0)_n^i=0$ for $i < n$, so by \eqref{eq:ric0i} applied to the extension of $(N',g')$ by $D'$ we obtain $\div' D'=0$.

The metric $g^D$ is locally obtained as the result of two consecutive extensions from the metric $g'$ on $N'$, first by $D'$ and then by $D$. Note that these two extensions ``commute", so we can first extend the Ricci flat metric $(N',g')$ by the identity endomorphism to obtain an Einstein manifold $\tilde N$ with the Einstein constant $-1$ (compare to Lemma~\ref{l:scalar}\eqref{it:sc1}), and then extend again by the endomorphism $\tilde D$ which coincides with $D'$ on $TN'$, is zero on $\partial_u$, and is Lie parallel along $\partial_u$, to obtain the Einstein metric $(M, g^D)$. By Theorem~\ref{th:Dconst}, the eigenvalues of $\tilde D$ are constant, and so the eigenvalues of $D'$ are also constant, as claimed.
\end{proof}

\begin{remark} \label{rem:10}
  Note that the condition for $(M,g^D)$ to be Einstein given in Theorem~\ref{th:10} is only necessary. To make it sufficient one has to additionally require that all the $D'$-deformations of the metric $(N',g')$ are Ricci flat, not only the metric $(N',g')$ itself. In general, it may be too difficult to classify Ricci flat deformations of a Ricci flat manifold, even under the additional assumption that $\Tr D' = 0$, $\Tr D'^2=n-1$, and $\div' D'=0$. One simple example is the Riemannian product of Ricci flat manifolds, with the operator $D'$ acting by scaling on each factor (compare to Example~\ref{ex:product}). When $\dim N' = 2$, this is the only possible case (see the end of the proof of Theorem~\ref{th:dim3} in the next section).
\end{remark}

\subsection{Proof of Theorem~\ref{th:12}} \label{ss:12} 
\eqref{it:12ein} $\Leftrightarrow$ \eqref{it:12Keta}. By Corollary~\ref{cor:Einext}, \eqref{it:12ein} is equivalent to the fact that equations \eqref{eq:mijj} are satisfied and that $\Ric^u$ given by \eqref{eq:Ricuij} satisfies \eqref{eq:ein1Rudetailes}. By Lemma~\ref{l:multn-1}, \eqref{eq:mijj} is equivalent to the fact that $S_n=0$ and $\mu_{in|n}=0$, for all $i < n$. Furthermore, we can assume that the frame $\overline{e}_1, \dots, \overline{e}_{n-1}$ is chosen in such a way that $\mu_{ni|j}=\mu_{nj|i}$ (and there is still the freedom of choosing it (locally) arbitrarily on a hypersurface transversal to $\overline{e}_n$).

From \eqref{eq:Ricuii} and \eqref{eq:ein1Rudetailes} by Lemma~\ref{l:multn-1} we get
\begin{equation*}
(\Ric^u)_n^n= -e^{-4u}\sum\nolimits_{k,l < n}(\mu_{nk|l})^2 +\tfrac14\sum\nolimits_{k,l<n}(\mu_{kl|n})^{2}=n-1,
\end{equation*}
which is equivalent to the fact that
\begin{equation*}
  (\Ric^0)_n^n = n-1 \quad \text{and} \quad \mu_{nk|l}=0,
\end{equation*}
for all $k, l < n$. Then from \eqref{eq:structure} and \eqref{eq:gijklie} we obtain $d\overline{\theta}^i=-\frac12 \sum_{j,k<n} \mu_{jk|i} \overline{\theta}^j \wedge \overline{\theta}^k$. It follows that $d(\sum_{j,k<n} \mu_{jk|i} \overline{\theta}^j \wedge \overline{\theta}^k)=0$, so $\overline{e}_n(\mu_{jk|i})=0$, for all $j,k < n$, and in particular, $\overline{e}_n(S_i)=0$.

Now computing the components $(\Ric^u)_i^i, \; (\Ric^u)_i^n, \; i <n$, by \eqref{eq:Ricuij} and \eqref{eq:Ricuii} and using \eqref{eq:ein1Rudetailes} we get
\begin{gather*}
(\Ric^u)_i^i = e^{-2u} ((\Ric^0)_i^i + \tfrac12 \sum\nolimits_{k<n}(\mu_{ik|n})^{2}) - \tfrac12 \sum\nolimits_{k<n}(\mu_{ik|n})^{2} = -2,
\\
(\Ric^u)_i^n = e^{-u} (\Ric^0)_i^n = 0,
\end{gather*}
which is equivalent to the fact that for all $i < n$,
\begin{equation*}
  (\Ric^0)_i^i = -2, \; (\Ric^0)_i^n = 0, \quad \text{and} \quad \sum\nolimits_{k<n}(\mu_{ik|n})^{2}=4.
\end{equation*}

Now taking $\xi=\overline{e}_n$ we find that $\xi$ is geodesic if and only if $\mu_{in|n}=0$. Furthermore, choosing an orthonormal frame in the distribution $\xi^\perp$ as in Lemma~\ref{l:multn-1}\eqref{it:multn-1sym} (so that $\mu_{ni|j}=\mu_{nj|i}$ for $i,j <n$) and defining $J=-\nabla \xi$ we obtain $\<J\overline{e}_i, \overline{e}_j\> = \mu_{ni|j}+\frac12\mu_{ij|n}$ for $i, j < n$, and so $\xi$ is Killing if and only if $\mu_{ni|j}=0$. Then the condition that the contact structure defined by $\xi$ is K-contact is equivalent to the fact that $\sum\nolimits_{k<n}(\mu_{ik|n})^{2}=4$, for all $i < n$. Finally, the condition that $(N,g)$ is $\eta$-Einstein, with $\ric = -2 g + (n+1) \eta \otimes \eta$, where $\eta=\overline{\theta}^n$, is equivalent to the fact that $(\Ric^0)_i^i = -2, \; (\Ric^0)_i^n = 0$ and $(\Ric^0)_n^n = n-1$, as the orthonormal basis $\overline{e}_i$ at a point can be chosen arbitrarily.

\smallskip

\eqref{it:12Keta} $\Leftrightarrow$ \eqref{it:12KahlerRf}. The foliation defined by $\xi$ on $(N,g)$ is geodesic. Locally take $N'=N/\xi$ and define the metric $g'$ on $N'$ in such a way that the natural projection is a submersion (this is possible as $\xi$ is Killing). Then the restriction of $J$ to $N'$ defines an almost K\"{a}hler structure. The fact that $(N',g')$ is Ricci flat follows from \cite[Equation~7.3]{BG} or by a direct calculation. The implication \eqref{it:12KahlerRf} $\Leftrightarrow$ \eqref{it:12Keta} is proved by reversing the construction.

\smallskip

Finally, it is easy to see that $(M,g^D)$ is almost K\"{a}hler, with the fundamental 2-form $e^{2u}(2du \wedge (dt+\theta') + \omega)$, in the notation of \eqref{it:12Keta}.

\section{Four-dimensional Einstein extensions}
\label{s:dim3}

In this section we consider the case $n=3$, the lowest dimension when our construction provides interesting examples. Note that in the case $n =2$, there are only two independent connection components, $\overline{\G}^1_{21}$ and $\overline{\G}^1_{22}$, and \eqref{eq:mijj} implies that $\overline{\G}^1_{21}(p_1-p_2) = \overline{\G}^1_{22}(p_1-p_2) = 0$, so either $(N,g)$ is flat, or $D$ is scalar, which again implies that $(N,g)$ is flat by Lemma~\ref{l:scalar}\eqref{it:sc1}. Then the manifold $(M,g^D)$ is hyperbolic.

\begin{proof}[Proof of Theorem~\ref{th:dim3}]
We consider all the possible eigenvalue types of $D$.

In the case when $D$ is scalar, the manifold $(N, g)$ is flat by Lemma~\ref{l:scalar}\eqref{it:sc1}. We can locally introduce Cartesian coordinates $x^1, x^2,x^3$ on $N$ and set $\overline{e}_i=\partial_i$. Then $N$ is abelian, $D$ is left-invariant, and its value at the identity of $N$ is obviously a derivation of the abelian Lie algebra $\gn$ of $N$. We get the first two rows of Table~\ref{t:d3}, up to scaling.

Next suppose that two out of three eigenvalues $p_i$ are zeros. Up to scaling, we can assume that $p_1=p_2=0, \; p_3=1$. By Theorem~\ref{th:01} we can choose local coordinates on $N$ in such a way that $\overline{ds^2}=(dx^3)^2 + (dx^1)^2 + e^{2x^1}(dx^2)^2$, and the orthonormal frame of eigenvectors of $D$ is $\overline{e}_1=\partial_1, \, \overline{e}_3=\partial_3$, and $\overline{e}_2= e^{-x^1} \partial_2$. Then $N$ is locally a solvable Lie group, with the only nontrivial relation in $\gn$ being $[\overline{e}_1,\overline{e}_2]=-\overline{e}_2$. Moreover, $D$ is left-invariant and is a derivation of $\gn$. Up to relabelling we obtain the case in the last row of Table~\ref{t:d3}, with $p=0$.

Suppose that $D$ is non-scalar and nonsingular. Up to scaling, we get $D=\diag(1,1,2)$ by Remark~\ref{rem:34}. Then by Theorem~\ref{th:12}, we can choose local coordinates on $N$ in such a way that $\overline{ds^2}=ds'^2+(dx^3+\theta')^2$, where $ds'^2=(dx^1)^2+(dx^2)^2$ is a two-dimensional flat metric and $\theta'=x^1 dx^2-x^2 dx^1$. Then $\overline{ds^2}=(dx^1)^2+(dx^2)^2+(dx^3+x^1 dx^2-x^2 dx^1)^2$. An orthonormal frame of eigenvectors of $D$ can be chosen as $\overline{e}_1=\db_1+x^2\db_3, \, \overline{e}_2= \partial_2-x^1 \partial_3$, and $\overline{e}_3=-\partial_3$. Then $N$ is locally the Heisenberg Lie group, with the only nontrivial relation in $\gn$ being $[\overline{e}_1,\overline{e}_2]=2\overline{e}_3$. The endomorphism field $D$ is left-invariant and its value at the identity is a derivation of $\gn$. We obtain the case in the third row of Table~\ref{t:d3}. The fact that the extension $(M,g^D)$ is locally isometric to $\bc H^2$ is well-known (see e.g., \cite[Section~6.5]{Heb}).

Suppose that one of the eigenvalues of $D$ is zero and the other two are nonzero. Up to relabeling and scaling, we can take $p_1=1, \; p_3=0, \; p_2 =p$, where $|p| \ge 1$.

We have the following lemma.
{
\begin{lemma}\label{l:pnot1}
If $p \ne 1$, then the only nonzero $\mu_{ij|k}, \; i<j$, are $\mu_{23|2}= 1$ and $\mu_{13|1}=- p$.
\end{lemma}
\begin{proof}
As in Section~\ref{s:Heis}, we denote $S_i=\sum_k \mu_{ki|k}$. Consider three cases.

Let $p \ne -1, 1, 2$. Then for pairwise non-equal $i,j,k$, we have $p_i+p_j-p_k \notin \{0,1,p\}$, so $\mu_{ij|k}=0$, by  Lemma~\ref{l:scalar}\eqref{it:sc3}.
From \eqref{eq:mijj} we obtain that $\mu_{21|2}=\frac{1}{p}S_1, \; \mu_{31|3}=\frac{p-1}{p}S_1, \; \mu_{12|1}=pS_2, \; \mu_{32|3}=(1-p)S_2, \; \mu_{13|1}= \frac{p}{p-1}S_3, \; \mu_{23|2}=\frac{1}{1-p}S_3$. Collecting the similar terms of \eqref{eq:Ricuii} as in \eqref{eq:collect}, we obtain by \eqref{eq:equate}:
\begin{gather*}
\overline{e}_1(S_1)+S_1^2=\overline{e}_1(S_1)+\tfrac{1+(1-p)^2}{p^2} S_1^2=0, \quad \overline{e}_2(S_2)+S_2^2=\overline{e}_2(S_2)+(p^2+(1-p)^2)S_2^2=0, \\ \overline{e}_3(S_3)+S_3^2=(1-p)^2,\quad \overline{e}_3(S_3)+\tfrac{1+p^2}{(1-p)^2}S_3^2=p^2+1,
\end{gather*}
and so $S_1=S_2=0$ and $S_3=1-p$ (up to changing the sign of $\overline{e}_3$), and the claim follows.

Now suppose $p=2$. Then by Lemma~\ref{l:scalar}\eqref{it:sc3}, we have $\mu_{12|3}=\mu_{31|2}=0$. From \eqref{eq:mijj} we obtain that $\mu_{21|2}=\mu_{31|3}=\frac12 S_1, \; \mu_{12|1}=2S_2$, $\mu_{32|3}=-S_2, \; \mu_{13|1}=2S_3, \; \mu_{23|2}=-S_3$.
From \eqref{eq:Ricuii} and \eqref{eq:ein1Ru} we get:
\begin{gather*}
\overline{e}_1(S_1)+S_1^2+\mu_{23|1}^2=\overline{e}_1(S_1)+\tfrac12 S_1^2-\tfrac12\mu_{23|1}^2=0, \quad \overline{e}_2(S_2)+S_2^2=\overline{e}_2(S_2)+5S_2^2=0, \\ \overline{e}_3(S_3)+S_3^2=1,\quad \overline{e}_3(S_3)+5S_3^2=5,
\end{gather*}
which then implies that $S_1=S_2=\mu_{23|1}=0$ and $S_3=-1$ (up to changing the sign of $\overline{e}_3$), and the claim follows.

The last case is $p=-1$. By Lemma~\ref{l:scalar}\eqref{it:sc3}, we get $\mu_{23|1}=\mu_{31|2}=0$. From \eqref{eq:mijj} we obtain $\mu_{21|2}=-S_1, \; \mu_{31|3}=2S_1, \; \mu_{12|1}=-S_2, \; \mu_{32|3}=2S_2, \; \mu_{13|1}=\mu_{23|2}=\frac12 S_3$. Then equations \eqref{eq:Ricuii} and \eqref{eq:ein1Ru} give:
\begin{gather*}
\overline{e}_1(S_1)+S_1^2=\overline{e}_1(S_1)+5 S_1^2=0, \quad \overline{e}_2(S_2)+S_2^2=\overline{e}_2(S_2)+5S_2^2=0, \\ \overline{e}_3(S_3)+S_3^2+\mu_{12|3}^2=4,\quad 2\overline{e}_3(S_3)+S_3^2-\mu_{12|3}^2=4,
\end{gather*}
which implies that $S_1=S_2=0$ and $S_3^2+3\mu_{12|3}^2=4, \; \overline{e}_3(S_3)=2\mu_{12|3}^2$. So all the $\mu_{ij|k}, i < j$, except possibly $\mu_{13|1}=\mu_{23|2}=\frac12 S_3$ and $\mu_{12|3}$, vanish. Then $d\overline{\theta}^1 = -\frac12 S_3 \overline{\theta}^1 \wedge \overline{\theta}^3, \; d\overline{\theta}^2 =-\frac12 S_3 \overline{\theta}^2 \wedge \overline{\theta}^3, \; d\overline{\theta}^3=-\mu_{12|3}\overline{\theta}^1 \wedge \overline{\theta}^2$. Differentiating the last equation we get $0=(\overline{e}_3(\mu_{12|3})+\mu_{12|3}S_3)\overline{\theta}^1 \wedge \overline{\theta}^2 \wedge \overline{\theta}^3$, and so $\overline{e}_3(\mu_{12|3})=-\mu_{12|3}S_3$. But then differentiating the equation $S_3^2+3\mu_{12|3}^2=4$ along $\overline{e}_3$ and using the fact that $\overline{e}_3(S_3)=2\mu_{12|3}^2$ we obtain $S_3\mu_{12|3}^2=0$. It follows that $\mu_{12|3}=0, \; S_3=2$ (up to changing the sign of $\overline{e}_3$), and $\mu_{23|2}=\mu_{13|1}=1$.
\end{proof}
}

We return to the proof of the theorem. Suppose $p \ne 1$. Then it follows from Lemma~\ref{l:pnot1} that $d\overline{\theta}^1 = p \overline{\theta}^1 \wedge \overline{\theta}^3, \; d\overline{\theta}^2 =- \overline{\theta}^2 \wedge \overline{\theta}^3, \; d\overline{\theta}^3=0$, so we can choose local coordinates on $N$ such that $\overline{\theta}^3=dx^3, \overline{\theta}^1=e^{-p x^3}dx^1, \; \overline{\theta}^2=e^{x^3}dx^2$. Then the metric $g$ is locally given by $\overline{ds^2}=e^{-2p x^3}(dx^1)^2+e^{2x^3}(dx^2)^2+(dx^3)^2$. Furthermore, the eigenvectors $\overline{e}_i$ of $D$ satisfy the relations $[\overline{e}_3,\overline{e}_1]= p \overline{e}_1, \; [\overline{e}_3,\overline{e}_2]=- \overline{e}_2, \; [\overline{e}_1,\overline{e}_2]=0$. Hence $N$ is (locally) a solvable, non-nilpotent Lie group, and $D$ is left-invariant and is a derivation of $\gn$. The Einstein metric $g^D$ on $M$ is locally given by $ds^2=e^{2(u-p x^3)}(dx^1)^2+e^{2(pu+x^3)}(dx^2)^2 +(dx^3)^2 +du^2$, which is the Riemannian product of two hyperbolic planes of curvature $-(p^2+1)$, which can be seen by the change of variables $y_1=(p^2+1)^{-1/2}(u-px^3), \; y_2=(p^2+1)^{-1/2}(pu+  x^3)$.

The only remaining case to consider is $p=1$, so that the eigenvalues of $D$ are $p_1=p_2=1$, $p_3=0$. By Theorem~\ref{th:10}, $(N,g)$ is the extension of a flat two-dimensional manifold $(N',g')$ by a symmetric endomorphism $D'$ such that $\Tr D'=0, \Tr D'^2=2$, and $\div' D'=0$. Choosing local Cartesian coordinates $x^1, x^2$ on $N'$, we obtain $D'=\left(\begin{smallmatrix}a&b\\b&-a\end{smallmatrix}\right)$, for some functions $a$ and $b$ on $N'$, with $a^2+b^2=2$. The condition $\div' D'=0$ gives $\partial_{x^1} a + \partial_{x^2} b =\partial_{x^1} b - \partial_{x^2} a= 0$, so $a - \mathrm{i}b$ is a holomorphic function with a constant module. It follows that both $a$ and $b$ are constants, so choosing $\partial_{x^1}, \partial_{x^2}$ to be the unit eigenvectors of $D'$ we obtain $D'=\diag(-1,1)$. Then the metric $g$ on $N$ is locally given by $ds^2=(dx^3)^2+e^{-2x^3} (dx^1)^2+e^{2x^3} (dx^2)^2$. Choosing the unit eigenvectors of $D$ as $\overline{e}_1=e^{x^3}\partial_1, \; \overline{e}_2=e^{-x^3}\partial_2, \; \overline{e}_3=\partial_3$ we obtain $[\overline{e}_3,\overline{e}_1]= \overline{e}_1, \; [\overline{e}_3,\overline{e}_2]=- \overline{e}_2$, and $[\overline{e}_1,\overline{e}_2]=0$, so $N$ is a Lie group defined by the corresponding Lie algebra $\gn$ and $D$ is a derivation of $\gn$. By Theorem~\ref{th:10}, the Einstein metric $g^D$ on $M$ is given by $ds^2=du^2+(dx^3)^2+e^{2(u-x^3)} (dx^1)^2+e^{2(u+x^3)} (dx^2)^2$, as in the last row of Table~\ref{t:d3}, with $p=1$.
\end{proof}

\section{Extensions of a Lie group. Proof of Theorem~\ref{th:group}} 
\label{s:homo}

Suppose that $N$ is a Lie group, and both the metric $g$ and the endomorphism field $D$ are left-invariant. We will mostly work on the level of Lie algebras. We can take the vector fields $\overline{e}_i$ left-invariant. Then $\mu_{ij|k}=\overline{\<[\overline{e}_i,\overline{e}_j], \overline{e}_k\>}$  are constants and are the structure constants of the Lie algebra $\gn$ of $N$, and we have $\sum\nolimits_{k,l}\mu_{jk|l}\mu_{il|k} = \sum\nolimits_{k,l}\<\ad_j\overline{e}_k, \overline{e}_l\> \<\ad_i\overline{e}_l, \overline{e}_k\>= B(\overline{e}_i,\overline{e}_j)$, where $B$ is the Killing form of $\gn$. Moreover, $S_l=\sum_k\mu_{kl|k}=-\Tr \ad_l$ (where we abbreviate $\ad_{\overline{e}_i}$ to $\ad_i$). 

Then equation \eqref{eq:mijj} (which is equivalent to \eqref{eq:ein1div}) takes the form
\begin{equation}\label{eq:mijjhom}
\Tr(\ad_{DX}-\ad_X D)=0, \quad \text{for all }X \in \gn,
\end{equation}
and equations \eqref{eq:Ricuij} and \eqref{eq:ein1Rudetailes} give
\begin{equation}\label{eq:Ricuijhom}
\begin{split}
(\Ric^u)_i^j= &-\tfrac12e^{-u(p_i+p_j)}B(\overline{e}_i,\overline{e}_j) -\tfrac12 \sum\nolimits_{l}(e^{u(p_i-p_j-2p_l)}\mu_{lj|i}+ e^{u(p_j-p_i-2p_l)}\mu_{li|j})\Tr \ad_l \\
&+\tfrac14e^{u(p_i+p_j)}\sum\nolimits_{k,l}e^{-2u(p_l+p_k)}\mu_{kl|i}\mu_{kl|j} -\tfrac12e^{-u(p_i+p_j)}\sum\nolimits_{k,l}e^{2u(p_l-p_k)}\mu_{ik|l}\mu_{jk|l}\\
= &\; ((\Tr D) p_i-(\Tr D^2))\K_{ij}.
\end{split}
\end{equation}
From \eqref{eq:Ricuijhom} (or from \eqref{eq:scaludetailes}) we also obtain
\begin{equation}\label{eq:scaluhom}
\begin{split}
\scal^u & = -\sum\nolimits_{k}e^{-2up_k}((\Tr\ad_k)^2+\tfrac12B(\overline{e}_k,\overline{e}_k)) -\tfrac14\sum\nolimits_{i,k,l}e^{2u(p_i-p_l-p_k)}(\mu_{kl|i})^{2}\\ & = (\Tr D)^2 - n(\Tr D^2).
\end{split}
\end{equation}
One can rewrite equation \eqref{eq:Ricuijhom} in a different form using the action of the group $\mathrm{GL}(\gn)$ on the Lie bracket of $\gn$ as in \cite{Lsurv, Lstand,N} (for a similar approach, with $\mathrm{GL}(\gn)$ acting on the inner product, see \cite[Section~3]{Heb}). For the metric Lie algebra $\gn$, denote the Lie bracket by $\mu(X,Y):=[X,Y]$, and for $A \in \mathrm{GL}(\gn)$, define the new Lie bracket on the underlying Euclidean space $(\br^{n},\ipb)$ of $\gn$, keeping the inner product fixed, by $A.\mu(X,Y)= A\mu(A^{-1}X,A^{-1}Y)$. The resulting metric Lie algebra is isomorphic (but not, in general, isometric) to $(\gn, \ipb)$. In our case, taking $A=e^{uD}$ we obtain the metric Lie algebra $(\gn(u), \ipb)$ with the Lie bracket $e^{uD}.\mu(X,Y)=e^{uD}[e^{-uD}X,e^{-uD}Y]$ whose structure constants $\mu^u_{ij|k}$ are given by $\mu^u_{ij|k}=e^{u(p_k-p_i-p_j)}\mu_{ij|k}$, so equation \eqref{eq:Ricuijhom} takes the form $\Ric^{\gn(u)}=(\Tr D) D -(\Tr D^2)\id$, where $\Ric^{\gn(u)}$ is the Ricci operator of $(\gn(u), \ipb)$. Using \cite[sec~7.38]{Bes} (or \eqref{eq:Ricuijhom}) we obtain that for the Lie bracket $e^{uD}.\mu$, equation \eqref{eq:Ricuijhom} is equivalent to
\begin{equation}\label{eq:Ricgroup}
\begin{split}
    \overline{\<\Ric^{\gn(u)}X ,X\>} \mkern-7mu & \mkern7mu =-\tfrac12 B(e^{-uD}X,e^{-uD}X)-\overline{\<(e^{uD}\ad_{e^{-uD}H}e^{-uD})X,X\>}\\
    &+ \tfrac14 \sum\nolimits_{k,l}\overline{\<[e^{-uD}E_k, e^{-uD}E_l], e^{uD}X\>}^2 -\tfrac12 \Tr(\ad_{e^{-uD}X}^*e^{2uD} \! \ad_{e^{-uD}X}e^{-2uD})\\
    &= (\Tr D) \overline{\<DX,X\>}-(\Tr D^2)\overline{\|X\|}^2,
\end{split}
\end{equation}
for all $X \in \gn$, where $H \in \gn$, \emph{the mean curvature vector} of the unimodular ideal, is defined by $\overline{\<H,X\>}=\Tr\ad_X$, and $\{E_k\}$ is an arbitrary orthonormal basis for $(\gn, \ipb)$ (not necessarily a basis of eigenvectors of $D$).

From Theorem~\ref{th:Dconst} we obtain the following.

\begin{corollary} \label{cor:homo}
Suppose $(N,g)$ is a Lie group and both $g$ and $D$ are left-invariant. The extension $(M,g^D)$ is Einstein if and only if equations \eqref{eq:Ricgroup} \emph{(}or equivalently \eqref{eq:Ricuijhom}\emph{)} and \eqref{eq:mijjhom} are satisfied.
\end{corollary}

An immediate consequence of \eqref{eq:Ricgroup} is the fact that the Ricci tensor of the metric Lie algebra $(\gn(u), \ipb)$ must be independent of $u$. One obvious case when this happens is when $D$ is a derivation of $\gn$, as then $e^{uD}$ is an automorphism, and so $e^{uD}.\mu(X,Y) = e^{uD}[e^{-uD}X,e^{-uD}Y] = \mu(X,Y)$. In that case, the resulting Einstein manifold $(M,g^D)$ is a Lie group with a left-invariant metric. Moreover, assuming Alekseevsky Conjecture, the manifold $(M,g^D)$ must be an Einstein solvmanifold (if $\det D \ne 0$, this follows from the fact that $\gn$ is nilpotent \cite{Jac}).

However, $D$ is not necessarily a derivation. The simplest example is when $D=\id$. Then $(N,g)$ is Ricci flat by Lemma~\ref{l:scalar}\eqref{it:sc1}, hence is flat by \cite{AK}, hence $\gn=\gn_1 \ltimes \gn_2$, an (orthogonal) semidirect product of the abelian algebras $\gn_1$ and $\gn_2$, with $\gn_2$ acting on $\gn_1$ by commuting skew-symmetric endomorphisms \cite{AK, BB1}. Therefore the algebra $\gn$ is not necessarily abelian, while $D=\id$ can be a derivation only of an abelian Lie algebra. Note however that $(N,g)$ is isometric to an abelian group and the extension $(M, g^D)$ is a solvable group with the hyperbolic metric. In the proof of Theorem~\ref{th:group}\eqref{it:decomphom} below, we will see more complicated examples of the same phenomenon. However, we know \emph{no examples of non-homogeneous Einstein extensions of a Lie group with a left-invariant metric by a left-invariant $D$}.
Under some additional assumptions on the structure of $\gn$, as in Theorem~\ref{th:group}, the fact that the extension $(M, g^D)$ is Einstein forces it to be an Einstein solvmanifold.

\begin{proof}[Proof of Theorem~\ref{th:group}]
\eqref{it:bpos} For $q \in \br$, introduce the sets $\mathcal{S}_q=\{(k,l,i) \, : \,  k \ne l, p_l+p_k-p_i=q\}$ and $P_q=\{i \, : \, p_i=q\}$. Let $Q = \{ q \in \br \, : \, P_q \cup \mathcal{S}_q \ne \varnothing\}$. Then \eqref{eq:scaluhom} gives
\begin{equation*}
    \scal^u =  -\sum\nolimits_{q \in Q} e^{-2uq}\bigl(\sum\nolimits_{k \in P_q}((\Tr\ad_k)^2+\tfrac12B(\overline{e}_k,\overline{e}_k)) +\tfrac14\sum\nolimits_{(k,l,i) \in \mathcal{S}_q} \mu_{kl|i}^{2}\bigr).
\end{equation*}
It follows that for all $q \in Q \setminus \{0\}$, we get $\sum\nolimits_{k \in P_q}((\Tr\ad_k)^2+\tfrac12B(\overline{e}_k,\overline{e}_k)) +\tfrac14\sum\nolimits_{(k,l,i) \in \mathcal{S}_q} \mu_{kl|i}^{2}=0$. As by assumption $B \ge 0$, all the terms on the left-hand side are zeros. Hence $\mu_{kl|i}=0$, unless $p_l+p_k = p_i$ and $k \ne l$, and $\Tr \ad_k=B(\overline{e}_k,\overline{e}_k)=0$, unless $p_k=0$. The former fact implies that for all $k,l$ we have $D[\overline{e}_k,\overline{e}_l]-[D\overline{e}_k,\overline{e}_l]-[\overline{e}_k,D\overline{e}_l]=\sum_i (p_i-p_k-p_l) \mu_{kl|i}=0$, so $D$ is a derivation of $\gn$. It follows that the extension $(M,g^D)$ is a metric Einstein Lie group, whose Lie algebra $\g$ is the extension of $\gn$ by the derivation $D$.

To see that $\g$ is solvable, consider the Killing form $B_\g$ of $\g$. As for all $X \in \gn$ we have $\ad_X e_0 \in \gn$, it follows that $B_\g(X,X)=B(X,X) \ge 0$. Moreover, as $\Tr \ad_k=0$, unless $p_k=0$, we get $\Tr \ad_{DX}=0$ for all $X \in \gn$, and so $B_\g(X, e_0)=0$ by \eqref{eq:mijjhom}. As $B_\g(e_0,e_0) = \Tr D^2 \ge 0$, the Killing form $B_\g$ is nonnegative, hence $\g$ is solvable \cite[Remark~4.8(a)]{Heb}.

\medskip

\eqref{it:decomphom} \eqref{it:decompD} Denote $\dim \h=d$ and let $\h = \Span(\overline{e}_a \, : \, a = 1, \dots, d), \; \gm = \Span(\overline{e}_k \, : \, k = d+1, \dots, n)$. Throughout the proof, the indices $a, b, c$ range from $1$ to $d$, and the indices $k,l,s$, from $d+1$ to $n$.  Note that $\mu_{kl|a} = \mu_{ab|k}=0$.

%
{
\begin{lemma} \label{l:mukldi}
{\ }

\begin{enumerate}[label=\emph{(\alph*)},ref=\alph*]
  \item \label{it:mutriv}
  $\mu_{ab|c}=\mu_{ab|k}=\mu_{kb|c}=\mu_{kl|c}=\mu_{ak|c}=0$, for all $a,b,c \le d < k, l$.

  \item \label{it:muakl}
  If $\mu_{ak|l} \ne 0$ for $a \le d < k, l$, then
  \begin{itemize}
    \item
    either $p_a=0$ and $p_k=p_l$,
    \item
    or $p_a \ne 0$ and either $p_k=p_l$ and then $\mu_{ak|l}+\mu_{al|k} = 0$, or $p_l-p_k-p_a=0$.
  \end{itemize}

  \item \label{it:mukls}
  For all $k, l, s > d$, we have $\mu_{sl|k}=0$ unless $p_k-p_l-p_s=0$.

  \item \label{it:muQN}
  $\sum\nolimits_{a: p_a = q} \big(\sum\nolimits_{s: p_s = p_k} \mu_{as|k}\mu_{as|l} - \sum\nolimits_{s: p_s = p_l} \mu_{ak|s}\mu_{al|s} \big)=0$, for  all $k, l > d$ with $p_l-p_k =q \ne 0$.
\end{enumerate}
\end{lemma}
\begin{proof}
\eqref{it:mutriv} is obvious, as $\h$ is abelian and is orthogonal to the derived algebra of $\g$.

\smallskip

\eqref{it:muakl} Denote $Q:=\{p_l - p_k \, : \, k,l =d+1, \dots, n\}$.
Take $i = j= a$ in equation \eqref{eq:Ricuijhom} (so that $\overline{e}_i = \overline{e}_j \in \h$). If $p_a = 0$ we get by \eqref{it:mutriv}
\begin{equation*}
(\Ric^u)_a^a= -(\Tr D^2) = -\tfrac12 B(\overline{e}_a,\overline{e}_a) - \tfrac12 \sum\nolimits_{k,l} e^{2u(p_l-p_k)}\mu_{ak|l}^2,
\end{equation*}
and so $\mu_{ak|l}=0$ unless $p_l=p_k$. If $p_a \ne 0$ we obtain
\begin{equation*}
\begin{split}
(\Ric^u)_a^a= &\; (\Tr D) p_a-(\Tr D^2) = -\tfrac12e^{-2up_a}\sum\nolimits_{k,l}\mu_{ak|l}\mu_{al|k} -\tfrac12e^{-2up_a}\sum\nolimits_{k,l} e^{2u(p_l-p_k)}\mu_{ak|l}^2\\
= &-\tfrac12 \sum\nolimits_{q \in Q \setminus \{0, p_a\}} e^{2u(q-p_a)}\sum\nolimits_{k,l: p_l-p_k=q}\mu_{ak|l}^2
-\tfrac12 \sum\nolimits_{k,l: p_l-p_k=p_a} \mu_{ak|l}^2 \\
&-\tfrac12e^{-2up_a}\Big(\sum\nolimits_{k,l} \mu_{ak|l} \mu_{al|k} + \sum\nolimits_{k,l: p_l=p_k} \mu_{ak|l}^2\Big).
\end{split}
\end{equation*}
It follows that $\mu_{ak|l} = 0$ unless $p_l-p_k \in \{0, p_a\}$. But then the expression in the last brackets equals $\frac12 \sum\nolimits_{k,l: p_l=p_k} (\mu_{ak|l}+\mu_{al|k})^2$ which implies that $\mu_{ak|l}+\mu_{al|k} = 0$ when $p_l=p_k$.

\smallskip

\eqref{it:mukls} Take $i = j = k$ (so that $\overline{e}_i = \overline{e}_j \in \gm$) in \eqref{eq:Ricuijhom}. Using \eqref{it:mutriv} we get
\begin{equation*}
\begin{split}
(\Ric^u)_k^k= &\; (\Tr D) p_k-(\Tr D^2) = - \sum\nolimits_{a} e^{-2u p_a} \mu_{ak|k} \Tr \ad_a \\
&+\tfrac12 e^{2u p_k} \sum\nolimits_{a,l} e^{-2u(p_l+p_a)} \mu_{al|k}^2 - \tfrac12 e^{-2u p_k} \sum\nolimits_{a,l} e^{2u(p_l-p_a)} \mu_{ak|l}^2\\
&+\tfrac14 e^{2u p_k}\sum\nolimits_{s,l} e^{-2u(p_l+p_s)}\mu_{sl|k}^2 - \tfrac12 e^{-2u p_k} \sum\nolimits_{s,l} e^{2u(p_l-p_s)} \mu_{ks|l}^2.\\
\end{split}
\end{equation*}
But the sum of the first three terms on the right-hand side is constant (does not depend on $u$) by \eqref{it:muakl}, and therefore the sum of the last two terms must also be a constant. Summing up these sums by $k$ we get $-\tfrac14 \sum\nolimits_{k,s,l} e^{2u(p_k-p_l-p_s)} \mu_{sl|k}^2$, and the claim follows.

\smallskip

\eqref{it:muQN} Let $k,l > d$ with $p_l-p_k=q \ne 0$. Using \eqref{it:mutriv}, \eqref{it:muakl} and \eqref{it:mukls} we obtain
\begin{equation*}
\begin{split}
0=(\Ric^u)_k^l= & \; \tfrac12e^{u(p_k+p_l)}\sum\nolimits_{a,s}e^{-2u(p_s+p_a)}\mu_{as|k}\mu_{as|l} -\tfrac12e^{-u(p_k+p_l)}\sum\nolimits_{a,s}e^{2u(p_s-p_a)}\mu_{ak|s}\mu_{al|s}\\
= & \; \tfrac12e^{u(p_k+p_l)}\sum\nolimits_{a,s: (p_a, p_s) \in \{(q, p_k), (-q,p_l)\}}e^{-2u(p_s+p_a)}\mu_{as|k}\mu_{as|l} \\
& \; - \tfrac12e^{-u(p_k+p_l)}\sum\nolimits_{a,s: (p_a, p_s) \in \{(-q,p_k),(q,p_l)\}}e^{2u(p_s-p_a)}\mu_{ak|s}\mu_{al|s}\\
= & \; \tfrac12e^{-uq} \sum\nolimits_{a: p_a = q} \Big(\sum\nolimits_{s: p_s = p_k} \mu_{as|k}\mu_{as|l} - \sum\nolimits_{s: p_s = p_l} \mu_{ak|s}\mu_{al|s} \Big)\\
& \; + \tfrac12e^{uq}\sum\nolimits_{a: p_a = -q} \Big(\sum\nolimits_{s: p_s = p_l} \mu_{as|k}\mu_{as|l} - \sum\nolimits_{s: p_s = p_k} \mu_{ak|s}\mu_{al|s}\Big),
\end{split}
\end{equation*}
and the claim follows.
\end{proof}
}

The claim of Lemma~\ref{l:mukldi}\eqref{it:mukls} is equivalent to the fact that the restriction of $D$ to $\gm$ is a derivation. The restriction of $D$ to $\h$ is also a derivation, as $\h$ is abelian. But $D$ may fail to be a derivation of the whole algebra $\gn$, as by Lemma~\ref{l:mukldi}\eqref{it:muakl}, the expression $(p_l-p_k-p_a) \mu_{ak|l}$ is not necessarily zero. To ``fix" that we will modify $\gn$ by a twisting, but first we will further clarify the action of $\ad_\h$ on $\gm$.

Let $\{q_1, \dots, q_m\} = \{p_{d+1}, \dots, p_n\}$ be the eigenvalues of the restriction of $D$ to $\gm$ labelled in such a way that $q_1 < \dots < q_m$, and let $d_\a, \; \a=1, \dots, m$, be the multiplicity of $q_\a$. Specify the basis $\overline{e}_k, \; k =d+1, \dots, n$, for $\gm$ in such a way that $p_{d+1} = \dots = p_{d+d_1} = q_1, \; p_{d+d_1+1} = \dots = p_{d+d_2} = q_2, \dots, p_{n-d_m+1} = \dots = p_n = q_m$. For $p_a=0$, denote $T_a$ the matrix of the restriction of $\ad_{\overline{e}_a}$ to $\gm$ relative to the chosen basis for $\gm$ (note that $\ad_{\overline{e}_a}$ acts trivially on $\h$). We have $(T_a)_{kl}=\mu_{al|k}$, and so by Lemma~\ref{l:mukldi}\eqref{it:muakl}, the matrix $T_a$ is block-diagonal with the diagonal blocks having dimensions $d_1 \times d_1, \, d_2 \times d_2, \dots, d_m \times d_m$, in that order (so that $T_a$ commutes with $D_{|\gm}$).

If $p_a \ne 0$, then $\ad_{\overline{e}_a}$ still acts trivially on $\h$. For the restriction of $\ad_{\overline{e}_a}$ to $\gm$, relative to the chosen basis for $\gm$, we have $(\ad_{\overline{e}_a})_{kl} = \mu_{al|k}$, and so by Lemma~\ref{l:mukldi}\eqref{it:muakl}, $(\ad_{\overline{e}_a})_{|\gm} = Q_a + N_a$, where $Q_a$ is a block-diagonal skew-symmetric matrix whose diagonal blocks have dimensions $d_1 \times d_1, \, d_2 \times d_2, \dots, d_m \times d_m$, in that order (so that $Q_a$ commutes with $D_{|\gm}$), and $N_a$ is a strictly upper or lower triangular matrix (depending on the sign of $p_a$) which may only have nonzero entries in the blocks $d_\a \times d_\b$ such that $q_\a-q_\b=p_a$ (so that $[D_{|\gm}, N_a] = p_a N_a$). In terms of the $\mu$'s, when $p_a \ne 0$, we have
\begin{equation}\label{eq:QNmu}
\begin{gathered}
  (Q_a)_{kl}=\mu_{al|k}, \quad Q_a^t = -Q_a, \quad (Q_a)_{kl} \ne 0 \,\Rightarrow \, p_k=p_l, \\
  (N_a)_{kl}=\mu_{al|k}, \quad (N_a)_{kl} \ne 0 \,\Rightarrow \, p_k=p_l+p_a.
\end{gathered}
\end{equation}
Let $d_0 \ge 0$ be the multiplicity of the eigenvalue $0$ of the restriction of $D$ to $\h$. Relabel the basis $\overline{e}_a, \; a =1, \dots, d$, for $\h$ in such a way that $p_1=\dots=p_{d_0}=0$ and $p_a \ne 0$ for $d_0 < a \le d$. We have the following lemma.

{
\begin{lemma} \label{l:allcommute}
All the matrices $T_a, Q_b, N_b$, where $1 \le a \le d_0, \; d_0 < b \le d$, pairwise commute.
\end{lemma}
\begin{proof}
As $\h$ is abelian, the operators $\ad_{\overline{e}_a}$ commute. Then $[T_a, T_b] = 0$, for all $a, b \le d_0$. Moreover, for all $a, b$ such that $1 \le a \le d_0 < b \le d$, we have $[T_a, Q_b + N_b] = 0$, and so $[T_a, Q_b] = [T_a, N_b] = 0$, because $T_a$ and $Q_b$ are block-diagonal, but all the nonzero blocks of $N_b$ are outside the diagonal.

Furthermore, for $a, b > d_0$, we have $0 = [Q_a+N_a, Q_b+N_b] = [Q_a, Q_b] + [Q_a, N_b] + [N_a, Q_b] +  [N_a, N_b]$. The same argument on the block structure of the $Q_a$'s and $N_a$'s now implies that $[Q_a, Q_b] = [Q_a, N_b] = [Q_b, N_a] = [N_a, N_b] = 0$, with only two possible exceptions:
\begin{enumerate}[label=(\Alph*),ref=\Alph*]
  \item \label{it:-}
  $p_a = - p_b \ne 0$; then $[N_a, N_b]$ is block-diagonal and so we only get $[Q_a, N_b] = [Q_b, N_a] = 0$ and $[Q_a, Q_b]+[N_a,N_b]=0$.
  \item \label{it:+}
  $p_a = p_b \ne 0$; then $[Q_a, N_b]$ and $[N_a, Q_b]$ have nonzero blocks at the same places and so we only get $[Q_a, Q_b] = [N_a, N_b] = 0$ and $[Q_a, N_b]+[N_a,Q_b]=0$.
\end{enumerate}
Consider case \eqref{it:+} first. Denote $q=p_a = p_b \ne 0$, and let $k, l > d$ be such that $p_l-p_k=q$ (if no such pair $(k,l)$ exists, then $N_a=0$, for all $a$ with $p_a=q$, by Lemma~\ref{l:mukldi}\eqref{it:muakl}). Then by Lemma~\ref{l:mukldi}\eqref{it:muQN} and from \eqref{eq:QNmu} we get
\begin{equation*}
\begin{split}
0 &=\sum\nolimits_{a: p_a = q} \Big(\sum\nolimits_{s: p_s = p_k} \mu_{as|k}\mu_{as|l} - \sum\nolimits_{s: p_s = p_l} \mu_{ak|s}\mu_{al|s} \Big) \\
&= \sum\nolimits_{a: p_a = q} \Big(\sum\nolimits_{s: p_s = p_k} (Q_a)_{ks} (N_a)_{ls} - \sum\nolimits_{s: p_s = p_l} (Q_a)_{sl} (N_a)_{sk} \Big)\\
&= \sum\nolimits_{a: p_a = q} ((Q_aN_a^t)_{kl} - (N_a^tQ_a)_{kl}),
\end{split}
\end{equation*}
as $(Q_a)_{ks} = 0$ when $p_s \ne p_k$ and $(Q_a)_{sl} = 0$ when $p_s \ne p_l$, by \eqref{eq:QNmu}. But then $\sum\nolimits_{a: p_a = q} [Q_a, N_a^t]_{kl} = 0$, and since all the entries $[Q_a, N_a^t]_{kl}$ with $p_l-p_k \ne q$ are zeros from \eqref{eq:QNmu} we get $\sum\nolimits_{a: p_a = q} [Q_a, N_a^t] = 0$, that is, $\sum\nolimits_{a: p_a = q} [Q_a, N_a] = 0$, as the matrices $Q_a$ are skew-symmetric.

Now take the commutator of the latter equation with $Q_b$ such that $p_b=q$. As in our case $[Q_a, Q_b] = 0$ and $[Q_a, N_b]=[Q_b,N_a]$ we obtain $0=\sum\nolimits_{a: p_a = q} [Q_b,[Q_a, N_a]] = \sum\nolimits_{a: p_a = q} [Q_a,[Q_b, N_a]] = \sum\nolimits_{a: p_a = q} [Q_a,[Q_a, N_b]]$. Multiplying by $N_b^t$ and taking the trace we obtain $0= \sum\nolimits_{a: p_a = q} \Tr([Q_a,[Q_a, N_b]] N_b^t)= -\sum\nolimits_{a: p_a = q} \Tr([Q_a, N_b] [Q_a,N_b]^t)$, as the matrices $Q_a$ are skew-symmetric. It follows that $[Q_a, N_b]=0$, for all $a, b$ such that $p_a=p_b \ne 0$.

\smallskip

Now consider case \eqref{it:-}. We have $[Q_a, Q_b]+[N_a,N_b]=0$. Taking the commutator with $Q_b$ and using the fact that $Q_b$ and $N_a$ commute we get $0=[Q_b,[Q_a, Q_b]]+[Q_b,[N_a,N_b]]=[Q_b,[Q_a, Q_b]]+[N_a,[Q_b,N_b]]=[Q_b,[Q_a, Q_b]]$, as $[Q_b,N_b]=0$ from case \eqref{it:+}. But then multiplying by $Q_a$ and taking the trace we obtain $\Tr ([Q_a, Q_b]^2)=0$, and so $[Q_a, Q_b]=0$.
\end{proof}
}
For $b=d_0+1, \dots, d$ define the operators $\mathcal{Q}_b, \mathcal{N}_b \in \End(\gn)$ by
\begin{equation*}
  \mathcal{Q}_b (\h) = \mathcal{N}_b (\h) = 0, \qquad \mathcal{Q}_b X = Q_b X, \; \mathcal{N}_b X = N_b X \; \text{for } X \in \gm.
\end{equation*}
Note that  $\ad_{\overline{e}_b} = \mathcal{Q}_b + \mathcal{N}_b$. Moreover, as both $D_{|\gm}$ and $(\ad_{\overline{e}_b})_{|\gm} = Q_b + N_b$ are derivations of $\gm$ and as $[D_{|\gm}, Q_b]=0,\; [D_{|\gm}, N_b] = p_b N_b, \; p_b \ne 0$, all the $Q_b$'s and $N_b$'s are derivations of $\gm$. It follows from Lemma~\ref{l:allcommute} that the operators $\ad_{\overline{e}_a}, \; a \le d_0$ and $\mathcal{Q}_b, \mathcal{N}_b, \; b > d_0$ are commuting derivations of the whole algebra $\gn$.

We now consider the metric solvable Lie algebra $\gn'$, defined on the same underlying linear space as $\gn$, with the same inner product $\ip'=\ipb$, and with the Lie bracket $[ \cdot, \cdot]'$ defined as follows:
\begin{equation*}
  \ad_{\overline{e}_k}'=\ad_{\overline{e}_k}, \; k > d, \qquad \ad_{\overline{e}_a}'=\ad_{\overline{e}_a}, \; a \le d_0, \qquad \ad_{\overline{e}_b}'=\mathcal{N}_b, \; d_0 < b \le d.
\end{equation*}
Then $\gn'$ is indeed a Lie algebra, and what is more, $D$ is a symmetric derivation of $(\gn',\ip')$. The algebra $\gn'$ is obtained from $\gn$ by the twisting $X \mapsto X + \phi(X)$, where $\phi$ is the homomorphism from $\gn$ to the Lie algebra of skew-symmetric derivations of $\gn$ defined on the basis by $\phi(\overline{e}_b) = -\mathcal{Q}_b$ for $d_0 < b \le d$, and $\phi(\overline{e}_k) = \phi(\overline{e}_a) = 0$ for $a \le d_0$ and $k > d$. It follows from \cite[\S~2]{Aneg} that the metric Lie groups $(N, g)$ and $(N', g')$ are isometric ($\gn'$ is a standard modification of $\gn$ \cite{GW}). Furthermore, as the field of endomorphisms $D$ defined on the underlying Riemannian space $(N, g)$ is $N$-left-invariant and as $[D,\phi(X)]=0$, for all $X \in \gn$, it is also $N'$-left-invariant.

\medskip

\eqref{it:decomphom} \eqref{it:ni1ab} We will prove that $D$ respects the decomposition $\gn=\h \oplus \gm$. Then the claim will follow from assertion \eqref{it:decompD}.

Denote $\dim \h=d$ and let $\h = \Span(E_a \, : \, a = 1, \dots, d), \; \gm = \Span(E_k \, : \, k = d+1, \dots, n)$. The index $a$ will range from $1$ to $d$, and the index $l$, from $d+1$ to $n$.

Polarising \eqref{eq:Ricgroup} with $u=0$ we obtain that for all $X \in \gm, \; Y \in \h$,
\begin{equation*}
\overline{\<\Ric^{\gn(0)}X,Y\>} = 0 = (\Tr D) \overline{\<DX,Y\>}.
\end{equation*}
If $\Tr D \ne 0$, we are done. Suppose $\Tr D = 0$. Then it follows from \eqref{eq:Ricgroup} that for each $u \in \br$, the metric Lie algebra $(\gn(u), \ipb)$ is Einstein, with the Einstein constant $-(\Tr D^2)$. The same is true for the metric Lie algebra $(\gn, \ipb_u)$ which is isometrically isomorphic to $(\gn(u), \ipb)$. Taking $u=0$ and $X \in \gm$ in \eqref{eq:Ricgroup} we obtain
\begin{equation*}
\begin{split}
    \overline{\<\Ric^{\gn(0)}X ,X\>} &= -\overline{\<\ad_{H}X,X\>} + \tfrac12 \sum\nolimits_{a,l}\overline{\<[E_a, E_l], X\>}^2 -\tfrac12 \Tr(\ad_{X}^*\ad_{X}) \\
    & = -\overline{\<\ad_{H}X,X\>} +\tfrac12 \sum\nolimits_{a,l}\overline{\<\ad_{E_a}^* X, E_l\>}^2 -\tfrac12 \sum\nolimits_{a} \overline{\<\ad_{X}^*\ad_{X} E_a,E_a\>} \\
    &= -\overline{\<\ad_{H}X,X\>} + \tfrac12 \sum\nolimits_{a}\overline{\|\ad_{E_a}^* X\|}^2 -\tfrac12 \sum\nolimits_{a} \overline{\|\ad_{E_a}X\|}^2 \\
    &= -(\Tr D^2)\overline{\|X\|}^2,
\end{split}
\end{equation*}
and so
\begin{equation} \label{eq:AH+AH'}
  -\tfrac12(A_H+A_H^*) +\tfrac12 \sum\nolimits_{a}[A_{E_a}, A_{E_a}^*]= -(\Tr D^2) \id_{\gm},
\end{equation}
where $A_Y, \; Y \in \h$, is the restriction of $\ad_Y$ to $\gm$.

In particular, taking the traces of both sides of \eqref{eq:AH+AH'} we obtain
\begin{equation} \label{eq:H^2}
    (n-d)\Tr(D^2)= \Tr A_H = \Tr \ad_H = \|H\|^2.
\end{equation}
If $\gn$ is unimodular, then $H=0$, and so $D=0$ (which is clearly a derivation). Moreover, $(N,g)$ is flat and $(M,g^D)$ is the Riemannian product of $(N,g)$ and the line. Suppose $\gn$ is non-unimodular. Then $(\gn, \ipb)$ is a standard metric solvable Einstein Lie algebra. In particular, by \cite[Corollary~4.11]{Heb}, the derived algebra of $\gn$ coincides with $\gm$, and by \cite[Theorem~4.10(1)]{Heb}, all the operators $\ad_Y, \; Y \in \h$, are normal. Then equation \eqref{eq:AH+AH'} gives
\begin{equation} \label{eq:AH+AH'new}
  \tfrac12(A_H+A_H^*) = (\Tr D^2) \id_{\gm}.
\end{equation}

Every algebra $(\gn, \ipb_u), \; u \in \br$, is a metric solvable non-unimodular Lie algebra, with the same Einstein constant $-\Tr(D^2)$. As any Einstein solvmanifold is standard by \cite{Lstand}, the $\ipb_u$-orthogonal complement $\h_u$ to $\gm=[\gn, \gn]$ must be abelian. We have $\h_u = e^{-2uD}\h$, and so $[e^{-2uD}X, e^{-2uD}Y]=0$, for all $X, Y \in \h$. Differentiating by $u$ at $u=0$ we obtain
\begin{equation}\label{eq:huabelian}
  [DX, Y] + [X, DY] = 0, \quad \text{for all} X, Y \in \h.
\end{equation}
Furthermore, equation~\eqref{eq:H^2} still holds if we replace $\ipb$ with $\ipb_u$ and $H$ with $H_u$, the mean curvature vector of $(\gn, \ipb_u)$ defined by $\<H_u, X\>_u = \Tr \ad_X$, for all $X \in \gn$. We have $\<e^{uD}H_u, e^{uD}Y\>=\<H,Y\>$ and so $H_u=e^{-2uD}H$. Then from~\eqref{eq:H^2} we get $(n-d)\Tr(D^2)=\|H_u\|^2_u=\<e^{-2uD}H,H\>$. Decomposing $H$ by an orthonormal basis of eigenvectors of $D$ we obtain $DH=0$. Then from \eqref{eq:huabelian} with $X=H$ we obtain $DY \in \Ker \ad_H$, for all $Y \in \h$. Clearly, $\h \subset \Ker\ad_H$, and if $X \in \gm \cap \Ker\ad_H$, then by \eqref{eq:AH+AH'new} we get $0= \tfrac12\<(A_H+A_H^*)X, X\> = (\Tr D^2) \|X\|^2$, and so $X = 0$, as $D \ne 0$. It follows that $D\h \subset \h$, and so $D$ preserves the decomposition $\gn=\h\oplus\gm$.
\end{proof}

	\bibliographystyle{amsplain} 
	\bibliography{oneext}

\def\cprime{$'$}
\providecommand{\bysame}{\leavevmode\hbox to3em{\hrulefill}\thinspace}
\providecommand{\MR}{\relax\ifhmode\unskip\space\fi MR }
\providecommand{\MRhref}[2]{%
  \href{http://www.ams.org/mathscinet-getitem?mr=#1}{#2}
}
\providecommand{\href}[2]{#2}
\begin{thebibliography}{10}

\bibitem{Aneg}
D.~V. Alekseevski{\u\i}, \emph{Homogeneous {R}iemannian spaces of negative
  curvature}, Mat. Sb. (N.S.) \textbf{96(138)} (1975), 93--117, 168.

\bibitem{AK}
D.~V. Alekseevski{\u\i} and B.~N. Kimel{\cprime}fel{\cprime}d, \emph{Structure
  of homogeneous {R}iemannian spaces with zero {R}icci curvature}, Funkcional.
  Anal. i Prilo\v Zen. \textbf{9} (1975), no.~2, 5--11.

\bibitem{BB1}
Lionel B{\'e}rard-Bergery, \emph{Sur la courbure des m\'etriques riemanniennes
  invariantes des groupes de {L}ie et des espaces homog\`enes}, Ann. Sci.
  \'Ecole Norm. Sup. (4) \textbf{11} (1978), no.~4, 543--576.

\bibitem{Bes}
Arthur~L. Besse, \emph{Einstein manifolds}, Ergebnisse der Mathematik und ihrer
  Grenzgebiete (3) [Results in Mathematics and Related Areas (3)], vol.~10,
  Springer-Verlag, Berlin, 1987.

\bibitem{Blair}
David~E. Blair, \emph{Riemannian geometry of contact and symplectic manifolds},
  second ed., Progress in Mathematics, vol. 203, Birkh\"auser Boston, Inc.,
  Boston, MA, 2010.

\bibitem{BG}
Charles~P. Boyer and Krzysztof Galicki, \emph{Einstein manifolds and contact
  geometry}, Proc. Amer. Math. Soc. \textbf{129} (2001), no.~8, 2419--2430.

\bibitem{BGbook}
\bysame, \emph{Sasakian geometry}, Oxford Mathematical Monographs, Oxford
  University Press, Oxford, 2008.

\bibitem{GW}
Carolyn~S. Gordon and Edward~N. Wilson, \emph{Isometry groups of {R}iemannian
  solvmanifolds}, Trans. Amer. Math. Soc. \textbf{307} (1988), no.~1, 245--269.

\bibitem{Heb}
Jens Heber, \emph{Noncompact homogeneous {E}instein spaces}, Invent. Math.
  \textbf{133} (1998), no.~2, 279--352.

\bibitem{Herv}
Sigbj{\o}rn Hervik, \emph{Einstein metrics: homogeneous solvmanifolds,
  generalised {H}eisenberg groups and black holes}, J. Geom. Phys. \textbf{52}
  (2004), no.~3, 298--312.

\bibitem{Jac}
N.~Jacobson, \emph{A note on automorphisms and derivations of {L}ie algebras},
  Proc. Amer. Math. Soc. \textbf{6} (1955), 281--283.

\bibitem{Lsurv}
Jorge Lauret, \emph{Einstein solvmanifolds and nilsolitons}, New developments
  in {L}ie theory and geometry, Contemp. Math., vol. 491, Amer. Math. Soc.,
  Providence, RI, 2009, pp.~1--35.

\bibitem{Lstand}
\bysame, \emph{Einstein solvmanifolds are standard}, Ann. of Math. (2)
  \textbf{172} (2010), no.~3, 1859--1877.

\bibitem{N}
Y.~Nikolayevsky, \emph{Einstein solvmanifolds and the pre-{E}instein
  derivation}, Trans. Amer. Math. Soc. \textbf{363} (2011), no.~8, 3935--3958.

\bibitem{NP}
Pawe{\l} Nurowski and Maciej Przanowski, \emph{A four-dimensional example of a
  {R}icci flat metric admitting almost-{K}\"ahler non-{K}\"ahler structure},
  Classical Quantum Gravity \textbf{16} (1999), no.~3, L9--L13.

\end{thebibliography}

\end{document}